\documentclass[11pt,a4paper]{amsart}
\usepackage{amsmath,amsfonts,amssymb,amsthm,amscd}

\newtheorem{thm}{Theorem}
\newtheorem{lem}[thm]{Lemma}

\newtheorem{defn}[thm]{Definition}

\newtheorem{rem}[thm]{Remark}

\newtheorem{ex}[thm]{Example}

\begin{document}

\title[$tt^{*}$-geometry on the big phase space]{$tt^{*}$-geometry on
the big phase space}
\date{\today}

\author{Liana David and Ian A.B. Strachan}

\address{Institute of Mathematics \lq\lq Simion Stoilow\rq\rq\, of the
Romanian Academy\\ Calea Grivitei no. 21, Sector 1 \\Bucharest\\
Romania}

\email{liana.david@imar.ro}

\address{~~}

\address{School of Mathematics and Statistics\\ University of Glasgow\\Glasgow G12 8QQ\\ U.K.}
\email{ian.strachan@glasgow.ac.uk}


\maketitle

{\bf Abstract:} The big phase space, the geometric setting for the
study of quantum cohomology with gravitational descendents, is a
complex manifold and consists of an infinite number of copies of
the small phase space. The aim of this paper is to define a
Hermitian geometry on the big phase space.

Using the approach of Dijkgraaf and Witten \cite{DijkgraafWitten},
we lift various geometric structures of the small phase space to
the big phase space. The main results of our paper state that
various notions from $tt^{*}$-geometry are preserved under such
liftings.

\tableofcontents

\section{Introduction}  The big phase space $M^\infty$ - the geometric arena for the study of quantum cohomology and topological quantum field theories with
gravitational descendants - consists of an infinite number of copies of the small phase space $M\,.$ Typically $M$ is the cohomology ring of some smooth projective
variety, so
\[
M^\infty = \prod_{n\geq 0} H^{*}(V;\mathbb{C})\,.
\]
The Poincar\'e pairing on $H^{*}(V;\mathbb{C})$ does not lift canonically to $M^{\infty}$ and certain lifts of this pairing that can be defined on the big phase space are highly degenerate \cite{liu4}. Thus
from a differential geometric point of view the big phase space is hard to study. However, in a talk given at the 2006 ICM, Liu \cite{liu3} defined a non-degenerate metric on $M^\infty$
which is a natural lift of the Poincar\'e pairing, namely
\[
{\hat{\eta}} \left( T^n(\gamma_\alpha),T^m(\gamma_\beta) \right) =
\delta_{mn} \eta(\gamma_\alpha,\gamma_\beta) = \delta_{mn}
\eta_{\alpha\beta}\,.
\]
Note that this is defined in terms of the Poincar\'e pairing on
$M$ and a certain endomorphism $T$, which encapsulates the
properties of the topological recursion operator.

The aim of this paper is to study Hermitian structures on $M^\infty\,.$ This will be achieved by coupling the above idea of Liu with original ideas of
Witten and Dijkgraaf \cite{DijkgraafWitten}, which relate
two-point correlator functions on $M$ to two-point
correlator functions on $M^\infty\,.$ As it turns out, this procedure
can be used to
lift structures such as the $tt^{*}$-equations of Cecotti
and Vafa \cite{cecotti} from
$M$ to $M^\infty .$
At the centre of the theory developed in this paper lies the following result:

\begin{thm}
Suppose that the $tt^{*}$-equations for a
pseudo-Hermitian metric $h$ and Higgs
field $C$ are satisfied on $M\,,$ so
\[
\partial^D C = 0 \,, \qquad \qquad {}^D\! R + [ C, C^\dagger] = 0\,.
\]
Then there exists a natural lift  of the  pseudo-Hermitian
metric and Higgs field, so that the $tt^{*}$-equations are
satisfied on $M^\infty\,.$
\end{thm}
\noindent
We also show that the Saito structure on $T^{1,0}M$ lifts to a Saito
structure on $T^{1,0}M^{\infty}$ and that various other substructures on $M$,
compatible with the Saito structure or governed by the $tt^{*}$-equations
(such as real Saito structures,  harmonic real Saito structures, harmonic potential real Saito
structures, $DChk$-structures and CV-structures),
can be similarly lifted to $M^{\infty}.$

\subsection{Background}
The study of Gromov-Witten invariants and intersection theory on the moduli space of curves has provided the impetus for
many recent developments in mathematics and mathematical physics. By studying integrals of products of $\psi$-classes over moduli spaces
\[
<\tau_{a_1} \ldots \tau_{a_n}>_g:=\int_{ {\overline{M}}_{g,n}} \psi_1^{a_1} \ldots \psi_n^{a_n}
\]
Witten \cite{witten1} derived three basic equations: the string
equation, the dilaton equation and the topological recursion
relation, and these, used recursively, enabled the invariants to
be constructed. Such invariants may be combined into a generating
function and it was conjectured by Witten, and later proved by
Kontsevich \cite{kontsevich}, that this is a certain solution of
the KdV hierarchy. These basic relations may then be generalized
and raised to the status of axioms of a topological quantum field
theory (TQFT).

Consider a smooth projective variety $V$ with $H^{\rm
odd}(V;\mathbb{C})=0$, $\{\gamma_1\,,\ldots\,,\gamma_N\}$ a basis
for the cohomology ring $M:=H^{*}(V;\mathbb{C})$ and let
$$
\eta_{\alpha\beta}  =  \eta(\gamma_\alpha,\gamma_\beta) =  \int_V
\gamma_\alpha \cup \gamma_\beta
$$
be the Poincar\'e pairing which defines a non-degenerate metric
which may be used to raise and lower indices. Following the
conventions of Liu and Tian \cite{liu1,liu2}, a flat coordinate
system $\{t^\alpha_0\,, \alpha=1\,,\ldots\,,N\}$ may be found on
$M$ so $\gamma_\alpha=\frac{\partial~}{\partial t^\alpha_0}$, and in
which the components of $\eta$ are constant.

The big phase space consists of an infinite number of copies of
the $M\,,$ the small phase space, so
\[
M^\infty = \prod_{n\geq 0} H^{*}(V;\mathbb{C})\,.
\]
The coordinate system $\{ t_{0}^{\alpha}\}$ induces, in a canonical way,
a coordinate system $\{t^\alpha_n\,,
n \in \mathbb{Z}_{\geq 0}\,,\alpha=1\,,\ldots\,,N\}$ on $M^{\infty}.$
We denote by $\tau_n(\gamma_\alpha) = \frac{\partial~}{\partial
t^\alpha_n}$ (also abbreviated to $\tau_{n,\alpha}\,)$
the associated fundamental vector fields, which
represent various tautological line bundles over the
moduli space of curves. A vector field
$\mathcal{W}=\sum_{m,\alpha} f_{m,\alpha} \tau_m(\gamma)$ is
called a primary field if $f_{m,\alpha}=0$ for $m>0$ and a
descendent field if $f_{0,\alpha}=0$, for any $\alpha .$

The descendent Gromov-Witten invariants
\[
< \tau_{n_1}(\gamma_{a_1}) \ldots \tau_{n_k}(\gamma_{a_k})>_g
\]
may be combined into generating functions, called prepotentials, labeled by the genus $g\,,$
\[
{\mathcal{F}}_g=\sum_{k\geq 0} \frac{1}{k!} \sum_{n_1 ,\alpha_1\ldots
n_k,\alpha_k} t^{\alpha_1}_{n_1} \ldots t^{\alpha_k}_{n_k} <
\tau_{n_1}(\gamma_{\alpha_1}) \ldots
\tau_{n_k}(\gamma_{\alpha_k})>_g\,,
\]
and these in turn may be used to define $k$-tensor fields on the big phase space, via the formula
\begin{equation}\label{k-point}
<< {\mathcal W}_{1}\cdots {\mathcal W}_{k}>>_{g} = \sum_{m_{1},\alpha_{1},\cdots , m_{k},\alpha_{k}} f^{1}_{m_{1},\alpha_{1}}
\cdots f^{k}_{m_{k}\alpha_{k}} \frac{\partial^{k} {\mathcal{F}_{g}}}{\partial t^{\alpha_{1}}_{m_{1}}\cdots \partial t^{\alpha_{k}}_{m_{k}} },
\end{equation}
for any vector fields ${\mathcal W}_{i} = \sum_{m,\alpha} f^{i}_{m,\alpha} \frac{\partial}{\partial t^{\alpha}_{m}}$. The tensor field (\ref{k-point}) has a physical interpretation
as the $k$-point correlation function of the TQFT.

The basic relationships between these correlators may then be encapsulated in the following:

\begin{defn}
Let $\tilde{t}^\alpha_n=t^\alpha_n - \delta_{n,1} \delta_{\alpha,1}$ and let
\begin{eqnarray*}
\mathcal{S} & = & -\sum_{n,\alpha} \tilde{t}^\alpha_n \tau_{n-1}(\gamma_{\alpha} )\,,\\
\mathcal{D} & = & -\sum_{n,\alpha} \tilde{t}^\alpha_n
\tau_{n}(\gamma_{\alpha})
\end{eqnarray*}
be the string and dilaton vector fields respectively. Then the prepotentials ${\mathcal{F}}_g$ satisfy the following relations:

\medskip
\noindent{\underline{String Equation:}}
\[
<< \mathcal{S} >>_g = \frac{1}{2} \delta_{g,0} \sum_{\alpha,\beta} \eta_{\alpha\beta} t^\alpha_0 t^\beta_0\,;
\]
\medskip
\noindent{\underline{Dilaton Equation:}}
\[
<< \mathcal{D} >>_g = (2g-2) {\mathcal{F}}_g - \frac{1}{24} \chi(V) \delta_{g,1}\,;
\]
\medskip
\noindent{\underline{Genus-zero Topological Recursion Relation:}}
\[
<<\tau_{m+1}(\gamma_\alpha) \tau_n(\gamma_\beta) \tau_k(\gamma_\sigma)  >>_{{}_0} = << \tau_{m}(\gamma_\alpha)\gamma_\mu >>_{{}_0}  <<\gamma^\mu \tau_n(\gamma_\beta) \tau_k(\gamma_\sigma) >>_{{}_0}\,.
\]
\end{defn}
The Topological Recursion Relation in turn leads to the generalized WDVV equation:
\[
\sum_{\mu,\nu} \frac{\partial^3 {\mathcal{F}}_0}{\partial
t^\alpha_m \partial t^\beta_n \partial t^\mu_0} \eta^{\mu\nu}
\frac{\partial^3 {\mathcal{F}}_0}{\partial t^\nu_0 \partial
t^\gamma_k \partial t^\delta_l} = \sum_{\mu,\nu} \frac{\partial^3
{\mathcal{F}}_0}{\partial t^\alpha_m \partial t^\gamma_k \partial
t^\mu_0} \eta^{\mu\nu} \frac{\partial^3 {\mathcal{F}}_0}{\partial
t^\nu_0 \partial t^\beta_n \partial t^\delta_l}.
\]
This may be written more succinctly by introducing the so-called quantum product between vector
fields on the big phase space:
\[
\mathcal{W}_1 \circ \mathcal{W}_2 = <<\mathcal{W}_1 \mathcal{W}_2
\gamma^\sigma >>_{{}_0}\, \gamma_\sigma\,
\]
where $\gamma^{\sigma}= \eta^{\sigma\beta}\gamma_{\beta}$ and
$(\eta^{\alpha\beta})$ is the inverse of $(\eta_{\alpha\beta}).$
With this the generalized WDVV equation just becomes the
associativity condition
$$
(\mathcal{W}_1\circ\mathcal{W}_2)\circ\mathcal{W}_3=\mathcal{W}_1\circ(\mathcal{W}_2\circ\mathcal{W}_3)\,.
$$

\medskip

By restricting such theories to primary vector fields with coefficients in the small phase space one recovers a Frobenius manifold structure \cite{dubrovin1,dubrovin2} on the small phase space,
with
\[
F_0(t_0^1\,,\ldots\,,t_0^N) = \left.\mathcal{F}_0( {\bf t})
\right|_{t^\alpha_n=0\,,\,n>0}
\]
becoming the prepotential for the Frobenius manifold
and multiplication given by
$$
\tau_{0,\alpha}\bullet\tau_{0,\beta} = <<\tau_{0,\alpha}\tau_{0,\beta}
\gamma^{\sigma}>>_{{}_0}\vert_{M} \gamma_{\sigma}.
$$
Frobenius
manifolds have turned out to be extremely ubiquitous structure
appearing, for example, via the work of K. Saito in singularity theory
\cite{saito} and in the theory of integrable systems, as well
as in quantum cohomology and mirror symmetry.

\medskip

The underlying manifolds, when studying Frobenius manifolds, are
actually complex manifolds, and the metric $\eta$ is an holomorphic (non-degenerate)  metric,  rather than a
real-valued metric \cite{dubrovin1}. To define real objects one
requires an anti-holomorphic involution which may then in turn be
used to define Hermitian objects. This direction of research was
started by Cecotti and Vafa \cite{cecotti} in their study of
$tt^{*}$-geometry (topological-anti-topological fusion). The idea
has since been developed by Dubrovin \cite{dubrovin3} (who studied
the integrability of $tt^{*}$-equations and developed the
connection with pluriharmonic maps), Hertling \cite{hert1} (who
connected $tt^{*}$-geometry with the work of Simpson on Higgs
bundles and generalizations of variations of Hodge structures) and
by Sabbah \cite{sabbah-art} (this stressing the actual
construction of these objects). For a collection of articles on
these subject, see \cite{wentland}.

\medskip

The historical development outlined above may be summarized in the following diagram:

\[
\begin{array}{ccc}
\left\{ \begin{array}{c} {\rm TQFT}\\ {\rm big~phase~space} \\ {\rm [Witten,~Dijkgraaf]} \end{array}\right\} & & \\
&&\\
\updownarrow
&&\\
&&\\
\left\{ \begin{array}{c} {\rm Frobenius~manifold}\\ {\rm small~phase~space} \\ {\rm [Dubrovin]} \end{array}\right\} & \longleftrightarrow &
\left\{ \begin{array}{c} tt^{*}-{\rm geometry, TERP-structures}\\ {\rm Hermitian-Higgs~bundles} \\ {\rm [Cecotti-Vafa,~Dubrovin,} \\{\rm Hertling,~Sabbah]} \end{array}\right\}
\\
\end{array}
\]
\medskip

\noindent As the title indicates, the purpose of this paper is to
introduce a Hermitian structure on the big phase space and to
study the properties of such a structure, in particular its
relationship with the standard, holomorphic, structures. It turns
out that one may define a full, infinite dimensional,
$tt^{*}$-geometry on the big phase space.

\medskip

The key object in our treatment is a certain endomorphism of
$T^{1,0}M^\infty$ introduced and studied in \cite{liu1,liu2},

\begin{equation}\label{endom-T}
T({\mathcal W}) := \tau_{+}({\mathcal W}) - {\mathcal S}\circ
\tau_{+}({\mathcal W}) ,\quad {\mathcal W}\in T^{1,0}M^{\infty},
\end{equation}
where
$$
\tau_{\pm} \left( \sum_{n,\alpha} f_{n,\alpha} \tau_{n,\alpha
}\right) := \sum_{n,\alpha} f_{n,\alpha} \tau_{n\pm 1,\alpha }\,.
$$
With this the Topological Recursion Relation takes the compact form
\begin{equation}\label{top-rec-rel}
T(\mathcal{W}_1) \circ \mathcal{W}_2 =0 \quad {\mathcal
W}_1\,,{\mathcal W}_2 \in T^{1,0}M^{\infty}\,.
\end{equation}
The Poincar\'e pairing is an holomorphic metric on the small phase
space, not the large phase space. An extension of this metric to
the big phase space was given in \cite{liu4}, namely
$<\mathcal{U},\mathcal{V}>=<<\mathcal{S}\mathcal{U}\mathcal{V}>>_{{}_0}\,,$
but this metric is degenerate, a fact that follows easily from the
use of the Topological Recursion Relation. However, in \cite{liu3}
a non-degenerate holomorphic metric on the big phase space was
defined, namely
\begin{equation}
{\hat{\eta}}(\mathcal{W},\mathcal{V}) = \sum_{k=0}^\infty << \mathcal{S} \tau^k_{-}(\mathcal{W}) \tau_{-}^k(\mathcal{V}) >>_{{}_0}\,.
\end{equation}
On using properties of the endomorphism $T$ it is easy to show that
\begin{equation}\label{hat-begin}
{\hat{\eta}} \left( T^n(\gamma_\alpha),T^m(\gamma_\beta) \right) =
\delta_{mn} \eta(\gamma_\alpha,\gamma_\beta) = \delta_{mn}
\eta_{\alpha\beta}\,.
\end{equation}
This last formula may be seen as a lift, using the endomorphism
$T$,  of $\eta$, defined on the small phase space, to the big
phase space. Combined with a basic result of Witten and Dijkgraaf
\cite{DijkgraafWitten} on the use of constitutive relations, this
gives a way to define Hermitian structures on $M^\infty$ starting
with finite dimensional structures on $M\,.$

\medskip

\subsection{Structure of Paper}

The rest of the paper is laid out as follows. Section
\ref{preliminary} is intended to fix notation. Here we briefly
recall the basic facts we need about the big phase space,
$tt^{*}$-geometry and the relations between $tt^{*}$-geometry and
Frobenius manifolds (see Definitions \ref{saito} and
\ref{diverse}).

 In Section \ref{lift-section} we develop our main tool from this
paper. We define the natural lift of functions and vector
fields from the small phase space $M$ to the big phase space
$M^{\infty}$ (see Definition \ref{def-ext}) and we study their basic
properties. The natural lift of a vector
field from $M$ to $M^{\infty}$ is a primary vector field, whose
coefficients in canonical flat coordinates are natural lifts
of the coordinates of the initial vector field. We study how
natural lifts of functions behave under derivations
and in particular we show that
any function on $M^{\infty}$, which is a natural lift of a
function on $M$, is annihilated by the vector fields from the
image of the operator $T$, defined by (\ref{endom-T}) (see Lemma
\ref{ian}). Lemma \ref{ian} is a main tool for
our computations from the next sections. Finally, using the same
ideas, we remark that any
tensor field $\mathcal F$ on $M$ may be lifted to a
tensor field
$\hat{\mathcal F}$ on $M^{\infty}$, referred as the natural
lift of $\mathcal F$ (see Section \ref{general}). In the
following sections we consider natural lifts of specific
tensor fields on $M$, and show they are part of
$tt^{*}$-structures on $M^{\infty}.$

In Section \ref{hermitian-str} we assume that the small phase
space comes with a real structure $k$, compatible with $\eta$
(i.e. $h:= \eta (\cdot , k \cdot )$ is a pseudo-Hermitian metric).
The natural lift $\hat{k}$ of $k$ is compatible with the
natural lift $\hat{\eta}$ of $\eta$ and we compute the Chern
connection of the associated pseudo-Hermitian metric
$\hat{h}=\hat{\eta}(\cdot , \hat{k}\cdot )$ on $T^{1,0}M^{\infty}$
and its curvature (see Lemmas \ref{chern} and \ref{curvature}) .
These computations will be used later on, in our study of the
$tt^{*}$-equations for the extended structures.

In Section \ref{higgs-str} we lift the Higgs field $C$ on $M$ to
a (non-commutative) Higgs field $\hat{C}$ on $M^{\infty}$ (see
Definition \ref{def-hat-c}). The field $\hat{C}$ is not the
natural lift of $C$ (in the sense of Section \ref{general}),
but it reflects that $2$-point functions lift trivially to the big
phase space (see Remark \ref{ian-notes}). The definition of
$\hat{C}$ also mirrors the properties of the Chern connection of
$\hat{h}$ and turns out to be very suitable for the
$tt^{*}$-geometry and the theory of Saito bundles on $M^{\infty}$.
Other lifts of the multiplication from the small phase space
to the big phase space (e.g. the quantum multiplication mentioned
above or the Liu's multiplication $\diamond$, see Remark
\ref{liu-multi}), appear in the literature, but they are not so
well suited for the $tt^{*}$-geometry on the big phase space.

In Section \ref{extended-str} we gather all lifted structures
defined in the previous sections and we prove that the basic
notions from the theory of Frobenius manifolds and
$tt^{*}$-geometry are preserved under these lifts. More
precisely, let $(T^{1,0}M, \nabla , \eta , C, R_{0}, R_{\infty})$
be the Saito bundle associated to the Frobenius manifold
$(M,\bullet , \eta , E)$, where $\nabla$ is the Levi-Civita
connection of $\eta$, $C_{X}(Y)=X\bullet Y$ is the Higgs field,
$R_{0}= C_{E}$ and $R_{\infty} =\nabla E$, where $E$ is the Euler
field. Assume that $k$ is a real structure on $M$, compatible with
$\eta$, and let $h:= \eta (\cdot , k\cdot )$ be the associated
pseudo-Hermitian metric. On $T^{1,0}M^{\infty}$ we consider
$\hat{\nabla}, \hat{\eta}, \hat{C}, \hat{R}_{0},
\widehat{R}_{\infty}, \hat{k}, \hat{h}$, where $\hat{\nabla}$ is a
flat connection on $T^{1,0}M^{\infty}$, defined by the condition
that all vector fields $T^{n}(\tau_{0,\alpha})$ ($n\geq 0$, $1\leq
\alpha\leq N$) are $\hat{\nabla}$-parallel, $\hat{\eta}$,
$\hat{R}_{0}$, $\widehat{R}_{\infty}$, $\hat{k}$ and $\hat{h}$ are
natural lifts of $\eta$, $R_{0}$, $R_{\infty}$, $k$, $h$, and
$\hat{C}$ is the lifted (non-commutative) Higgs field, as
defined in Section \ref{higgs-str} (see Definition
\ref{def-hat-c}). In a first stage we show that $(T^{1,0}M,
\hat{\nabla}, \hat{\eta}, \hat{C}, \hat{R}_{0}, \hat{R}_{\infty})$
is a Saito bundle. Then
we prove our main results from this paper: namely, we assume that
various compatibility conditions between the real structure $k$
and the Saito structure $(\nabla , \eta , C, R_{0}, R_{\infty})$
hold (giving rise to a harmonic Higgs bundle,  a real Saito bundle,
a harmonic potential real Saito bundle, a $DChk$-bundle or a CV-bundle), and we
prove that all such conditions are inherited by the lifted
structures on the big phase space (see Sections
\ref{harmonic-sect}, \ref{extended-sect},
\ref{extended-real-sect}).\\

{\bf Acknowledgements.} Part of this work was carried while L.D.
was a visitor, during April-June 2012, at Institut des Hautes
Etudes Scientifiques (France). Hospitality, excellent working
conditions and financial support are acknowledged. This work is
also partially supported by a grant of the Romanian National
Authority for Scientific Research, CNCS-UEFISCDI, project no.
PN-II-ID-PCE-2011-3-0362.

\section{Preliminary material}\label{preliminary}

This section is intended to recall basic facts we need from the
geometry of the big phase space, theory of Frobenius manifolds and
$tt^{*}$-geometry. The manifolds we consider are complex and the
vector bundles holomorphic. For a complex manifold $M$, we denote
by $T^{1,0}M$ the holomorphic tangent bundle of $M$, by ${\mathcal
T}_{M}$, ${\mathcal T}^{1,0}_{M}$, $C^{\infty}(M)$, $C^{\infty}(M,
\mathbb{R})$ the sheaf of holomorphic vector fields, vector fields
of type $(1,0)$ and complex, respectively real valued smooth
functions on $M$. Vector fields on the small phase space will be
denoted by $X$, $Y$, $Z$, etc, while vector fields on the big
phase space will be usually denoted by ${\mathcal U}, {\mathcal
V}$, ${\mathcal W}$. (Unless otherwise specified,  vector fields
are of type $(1,0)$, not necessarily holomorphic, and functions
are smooth and complex valued). For an holomorphic bundle
$V\rightarrow M$, $\Omega^{1}(M, V)$, $\Omega^{1,0}(M, V)$ and
$\Omega^{0,1}(M, V)$ will denote, respectively,
the sheaves of holomorphic $1$-forms,
forms of type $(1,0)$ and forms of type $(0,1)$ on $M$, with values in $V$.

\subsection{The geometry of the big phase space}\label{prel-big}

The material in this section is taken directly from
\cite{liu1,liu2,liu3} and no proofs will be given. The first lemma
shows that the string vector field behaves like a unit for the
quantum product restricted to primary fields, and the second lemma
derives properties of a naturally defined covariant derivative on
the big phase space.

\begin{lem}
For all primary vector fields $\mathcal{W}$ and for all vectors
fields $\mathcal{U}\,,\mathcal{V}$ on $M^{\infty}$,
\begin{eqnarray*}
\mathcal{S} \circ \mathcal{W} & = & \mathcal{W} \,,\\
\mathcal{S} \circ \mathcal{U} \circ \mathcal{V} & = & \mathcal{U} \circ \mathcal{V}\,.
\end{eqnarray*}
\end{lem}

\begin{rem}\label{liu-multi}{\rm Besides the quantum product, there are also other interesting multiplications on the big phase space, see e.g. \cite{liu3}.
The multiplication
$$
T^{n}({\mathcal U}) \diamond  T^{m}({\mathcal V}) = \delta_{mn}
T^{n}({\mathcal U}\circ {\mathcal V}),
$$
where ${\mathcal U}$ and ${\mathcal V}$ are primary vector fields,
is commutative, associative, with unit field
$$
\hat{\mathcal S}=\sum_{k\geq 0}T^{k} ({\mathcal S}\circ {\mathcal S}).
$$
One may also show that the metric $\hat{\eta}$ and the multiplication $\diamond$ are compatible in the sense that
\[
{\hat{\eta}} (\mathcal{W}_1 \diamond \mathcal{W}_2,\mathcal{W}_3) = {\hat{\eta}} (\mathcal{W}_1,\mathcal{W}_2 \diamond \mathcal{W}_3)\,.
\]
However one may check that $(M, \diamond , \hat{\mathcal S})$ is not an $F$-manifold \cite{hert0}.}
\end{rem}

\begin{lem}\label{cr}
Let $\nabla$ be the covariant derivative defined by
\[
\nabla_{\mathcal{V}} \mathcal{W} = \sum_{m,\alpha} \mathcal{V} \left( f_{m,\alpha}\right) \tau_m(\gamma_\alpha)
\]
where $\mathcal{W} = \sum_{m,\alpha} f_{m,\alpha}
\tau_m(\gamma_\alpha)$ and $\mathcal{V}$ is an arbitrary vector
field on $M^{\infty}$. Then
\begin{align}
\nonumber {\mathcal W} << {\mathcal W}_{1}\cdots {\mathcal W}_{k}
>>_{{}_0} & =  \sum_{i=1}^{k} << {\mathcal W}_{1}\cdots
{\mathcal W}_{i-1} (\nabla_{\mathcal W}{\mathcal W}_{i}){\mathcal
W}_{i+1}\cdots {\mathcal W}_{k} >>_{{}_0}\\
\label{deriv-tensor} & + << {\mathcal W}{\mathcal W}_{1}\cdots
{\mathcal W}_{k}>>_{{}_0},
\end{align}
for any vector fields ${\mathcal W}, {\mathcal
W}_{i}$ on $M^{\infty}$,  and
\begin{eqnarray*}
\nabla_\mathcal{V}\left( T^k(\mathcal{W})\right) &
= & T^k\left( \nabla_{\mathcal{V}} \mathcal{W} \right) - T^{k-1} \left( \mathcal{V}\circ\mathcal{W} \right)\,,\\
\left[ T^n(\gamma_\alpha),T^m(\gamma_\beta)\right] & = & 0 \,, \qquad n,m \geq 1\,,\\
\left[ T^n(\gamma_\alpha),\gamma_\beta\right] & = & T^{n-1} \left( \gamma_\alpha \circ \gamma_\beta \right) \,, \qquad n\geq 1\,.
\end{eqnarray*}
\end{lem}
\noindent  Thus the vector fields $\{ T^n(\gamma_\alpha) \,,
n\in\mathbb{Z}_{\geq 0}\,, \alpha=1\,,\ldots\,,N\}$ form a frame
for $T^{1,0}M^\infty\,,$ but not a coordinate frame, due to the
last of the above equations.

\begin{rem}{\rm The connection $\nabla$ as defined above
induces an holomorphic connection on $T^{1,0}M$ in the obvious way
(if $X$ and $Y$ are vector fields on $T^{1,0}M$ then $\nabla_{X}Y$
is also a vector field on $T^{1,0}M$),  for which $\{ \tau_{0,\alpha}\}$ are
flat. In particular, the induced connection
coincides with the Levi-Civita connection of $\eta$, and will
be also denoted by $\nabla .$ It will be clear from the context when
$\nabla$ acts on $T^{1,0}M$ or $T^{1,0}M^{\infty}.$}
\end{rem}

\subsection{Frobenius manifolds and $tt^{*}$-geometry}\label{frob}

In this section we recall the basic definitions from Frobenius
manifolds and $tt^{*}$-geometry, see e.g. \cite{hert1, sabbah-art}. We
begin with the definition of Saito bundles and we explain their
relation with Frobenius manifolds. Then we add a real structure on
a Saito bundle and we define various compatibility conditions
between the Saito structure and the real structure. We work in the
holomorphic category: all manifolds, bundles, tensor fields,
connections etc from this section are holomorphic, unless
otherwise stated.

\begin{defn}\label{saito} A \underline{Saito bundle} (of weight $w\in \mathbb{C}$) is a
vector bundle
$$
\left(\pi :V\rightarrow M,\nabla , g, C , R_{0}, R_{\infty}\right)
$$
endowed with a connection $\nabla$, a metric $g$, a
vector valued $1$-form $C \in \Omega^{1}(M, \mathrm{End}V)$ and
two endomorphisms $R_{0}$ and $R_{\infty}$, such that the
following conditions are satisfied:
\begin{equation}\label{saito0}
R^{\nabla}=0,\quad d^{\nabla}C =0,\quad \nabla g=0,\quad C\wedge C
=0,\quad C^{*}=C ,
\end{equation}
and
\begin{align*}
&\nabla R_{0}+ C = [C , R_{\infty}], \quad [R_{0},
C ]=0,\quad R_{0}^{*}= R_{0};\\
&\nabla R_{\infty}=0,\quad R_{\infty}^{*} + R_{\infty} =
-w\mathrm{Id}_{V}.
\end{align*}
Above $d^{\nabla}C$ and $C\wedge C$ are $\mathrm{End}(V)$-valued
$2$-forms, defined by
\begin{align*}
(d^{\nabla}C)_{X, Y}&:= \nabla_{X}\left( C_{Y}\right) -
\nabla_{Y}\left( C_{X}\right) - C_{[X, Y]},\\
(C\wedge C )_{X,Y} &:= C_{X}C_{Y}-C_{Y}C_{X}
\end{align*}
for any $ X, Y\in {\mathcal T}_{M}$, and the superscript ``$*$''
denotes the $g$-adjoint (in particular,
$C^{*}_{X}\in\mathrm{End}(V)$  is the $g$-adjoint of
$C_{X}\in\mathrm{End}(V)$). Moreover, $[R_{0}, C]$ is an
$\mathrm{End}(V)$-valued $1$-form, which, on $X\in T^{1,0}M$, is equal
to $[R_{0}, C_{X}].$
\end{defn}

A Frobenius manifold $(M, \bullet , e, \eta , E)$ defines a Saito
structure $(\nabla , \eta , C ,R_{0}, R_{\infty})$ on $T^{1,0}M$,
where $\nabla$ is the Levi-Civita connection of $\eta$, $C_{X}Y:=
X\bullet Y$ is the Higgs field, $R_{0}:= C_{E}$ is the
multiplication by the Euler field and $R_{\infty}:= \nabla E$. The
weight of this Saito structure is $d$, where $L_{E}(g)=-dg$.
Conversely, any Saito bundle whose rank is equal to the dimension
of the base, together with a suitably chosen parallel section
(usually called primitive homogeneous), gives rise to a Frobenius
structure on the base of the bundle \cite{saito} (see also
\cite{sabbah}).

We now add a real structure to a Saito bundle and define various
compatibility conditions, which give rise to the notions of real
Saito bundles, harmonic real Saito bundles and harmonic potential
real Saito bundles.

\begin{defn}\label{diverse} 1)
A \underline{real Saito bundle} is a Saito bundle
$$
\left(\pi :V\rightarrow M,\nabla , g,C,  R_{0}, R_{\infty}\right)
$$
together with a real structure $k:V\rightarrow V$ (i.e. $k$ is a
fiber-preserving smooth anti-linear involution) such that $g$, $k$
are compatible (i.e. $h:= g(\cdot , k\cdot )$ is a
pseudo-Hermitian metric) and $g$, $h$ are also compatible
(i.e. $D(g)=0$, where
$D$ is the Chern connection of  $h$).\\

2) A \underline{harmonic real Saito bundle} is a real Saito bundle
$$
\left(\pi :V\rightarrow M,\nabla ,  g, C, R_{0}, R_{\infty},
k\right)
$$
such that $(V, C,h = g(\cdot , k\cdot ))$ is a harmonic Higgs
bundle, i.e. the $tt^{*}$-equations
$$
(\partial^{D}C )_{Z_{1}, Z_{2}}=0,\quad R^{D}_{Z_{1}, \bar{Z}_{2}}
+ [C_{Z_{1}}, C^{\dagger}_{\bar{Z}_{2}}]=0, \quad Z_{1},Z_{2}\in
T^{1,0}M
$$
hold, where
$$
(\partial^{D}C )_{Z_{1},Z_{2}} := D_{Z_{1}}(C_{Z_{2}}) -
D_{Z_{2}}(C_{Z_{1}}) - C_{[Z_{1},Z_{2}]},
$$
$C^{\dagger}\in \Omega^{0,1}(M, \mathrm{End}V)$ denotes the
$h$-adjoint of $C\in\Omega^{1,0}(M, \mathrm{End}V)$ and $D$ is the Chern connection of $h$.\\

3) A \underline{harmonic potential real Saito bundle} is a harmonic real
Saito bundle
$$
\left(\pi : V\rightarrow M,\nabla ,g, C, R_{0}, R_{\infty},
k\right)
$$
together with a smooth $g$-self adjoint endomorphism $A$
of $V$ (called the potential) such
that the following conditions hold:\\

3a) $D^{(1,0)}A = C$, where $D$ is the Chern connection of $h:= g(\cdot , k\cdot )$;\\

3b) $D^{(1,0)}= \nabla - [A^{\dagger},C ]$, where
$A^{\dagger}$ is the $h$-adjoint of $A$;\\

3c) the endomorphism $R_{\infty} + [A^{\dagger}, R_{0}]$ is
$h$-self adjoint.

\end{defn}

A Frobenius manifold with a suitably chosen real structure and a (not necessarily holomorphic)
endomorphism of the tangent bundle gives
rise to the notion of harmonic Frobenius manifold, defined as
follows.

\begin{defn}
A \underline{harmonic Frobenius manifold} is a complex Frobenius
manifold $(M,\circ , e,g, E)$ such that $L_{E}(g) = dg$, with
$d\in\mathbb{R}$, together with a real structure $k$ on $T^{1,0}M$
and a smooth endomorphism $A$ of $T^{1,0}M$ (called
the potential of the Frobenius manifold), such that the associated
Saito bundle $(T^{1,0}M, \nabla, g,  C , R_{0} = C_{E}, R_{\infty}=\nabla
E )$ is a harmonic potential real Saito bundle, with real
structure $k$ and potential $A$.
\end{defn}

One may show that any harmonic potential real Saito bundle whose
rank is equal to the dimension of the base, together with a
parallel primitive real homogeneous section, gives rise to a
harmonic Frobenius structure on the base of the bundle. For our
purposes, we do not need this construction. For a precise
statement and proof, see Corollary 1.31 of \cite{sabbah-art}.

\begin{rem}\label{CV}{\rm
As will be proved in Section \ref{subsect-saito},
the tangent bundle of the big phase space comes
naturally equipped with a Saito structure, obtained by lifting  the Saito
structure of the small phase space. Thus  it is natural  to
look for $tt^{*}$-structures on the big phase space,  compatible
with this lifted Saito structure.   For this reason, in Definition
\ref{diverse} above only notions from $tt^{*}$-geometry, which admit an
underlying Saito structure, were recalled. It is worth to remark however that other important
notions exist in $tt^{*}$-geometry, which are build as an enrichment of the notion of
harmonic Higgs bundle, rather than Saito bundle, and do not necessarily admit an underlying
Saito structure.  For example, in the language of \cite{hert1}, one may consider
the notion of $DChk$-bundle, which is
a harmonic Higgs bundle $(V, C, h)$, together
with a real structure $k$, such that $g:= h(\cdot , k\cdot )$ is a symmetric (holomorphic)
metric, compatible with $h$. A richer notion is  the notion of   CV-bundle,
which is, by definition,  a $DChk$-bundle $(V, C, h, k)$, together with two endomorphisms
$\mathcal U$ and $\mathcal Q$ (the latter not necessarily holomorphic),
satisfying the following conditions:\\

i) for any $X\in T^{1,0}M$, $[C_{X}, {\mathcal U}]=0$;\\

ii) for any $X\in T^{1,0}M$,
$$
D_{X}{\mathcal U} + [C_{X}, Q] - C_{X}=0
$$
and
$$
D_{X}{\mathcal Q}  - [C_{X}, k{\mathcal U}k]=0
$$
(as usual, $D$ denotes the Chern connection of $h$).\\

iii) $\mathcal Q$ is $h$-self adjoint and $g$-skew adjoint.\\

Such structures arise naturally in singularity theory \cite{hert1}. While our primary interest
is in the structures presented in Definition
\ref{diverse}, it turns out that the notions of $DChk$ and $CV$-structures are preserved under our liftings to the big phase space. These facts will be referred to throughout the course of this paper, but no proofs will be given.  }
\end{rem}

\section{Natural lifts from $M$ to $M^{\infty}$}\label{lift-section}

In this section we describe a canonical way to lift tensor
fields from the small phase space $M$ to the big phase space
$M^{\infty}.$ This will be our main tool in the construction of a
Hermitian geometry on the big phase space.
The idea is an adoption
of the use of constitutive equations and was originally introduced
by Dijkgraaf and Witten \cite{DijkgraafWitten}. We preserve the
notation from the Introduction and Section \ref{prel-big}.

\subsection{Natural lifts of functions and vector fields}

\begin{defn}\label{def-ext} i) Let $f$ be a function on $M$. The function
$\hat{f}$ on $M^{\infty}$, defined by
\begin{equation}\label{ext}
\hat{f}(t^{\alpha}_{n}):= f\left( \eta^{\alpha\beta}<<\tau_{0,1}\tau_{0,\beta}>>_{{}_0}\right) ,
\end{equation}
is called the natural lift of $f$.\\

ii) Let $X\in {\mathcal T}^{1,0}_{M}$ a vector field on $M$, given in flat coordinates
$\{ t^{\alpha}_{0}\}$ by
$$
X= \sum_{\alpha =1}^{N}f^{\alpha}\tau_{0,\alpha} .
$$
The primary field
$$
\hat{X}:= \sum_{\alpha =1}^{N}\widehat{f^{\alpha}}\tau_{0,\alpha}
$$
is called the natural lift of $X$ to $M^{\infty}.$
\end{defn}

We now develop some simple properties of natural lifts, which
will be useful in the next sections.

\begin{lem}\label{ian}
For any  ${\mathcal W}\in {\mathcal T}^{1,0}_{M^{\infty}}$
and $f\in C^{\infty}(M)$,
\begin{equation}\label{x-f}
{\mathcal W}(\hat{f}) = \left[ \frac{\partial f}{\partial
t_{0}^{\beta}}\right]^{\wedge} \eta^{\beta\sigma} <<
\tau_{0,1}\tau_{0,\sigma}{\mathcal W}>>_{0}.
\end{equation}
In particular,
\begin{equation}\label{hat-kill}
 T({\mathcal W}) (\hat{f})=0.
\end{equation}
\end{lem}

\begin{proof}
Consider the vector valued
function
$$
\vec{u} := (u^{1},\cdots ,u^{N}),\quad
u^{\alpha}:= \eta^{\alpha\beta} <<\tau_{0,1}\tau_{0,\beta}>>_{{}_0}.
$$
From (\ref{ext}),  $\hat{f} = f\circ \vec{u}$
and
\begin{equation}\label{deriv-1}
\frac{\partial\hat{f}}{\partial t^{\alpha}_{n}}=
\left[\frac{\partial f}{\partial t^{\beta}_{0}}\right]^{\wedge}
\frac{\partial u^{\beta}}{\partial t^{\alpha}_{n}}=
\left[\frac{\partial f}{\partial t^{\beta}_{0}}\right]^{\wedge}
\eta^{\beta\sigma} <<
\tau_{0,1}\tau_{0,\sigma}\tau_{n,\alpha}>>_{{}_0},
\end{equation}
where we used (\ref{deriv-tensor}).
Relation (\ref{x-f}) follows.  Relation (\ref{hat-kill}) follows from (\ref{x-f}),  the
Topological Recusion Relation (\ref{top-rec-rel}) and the definition
of the quantum product.\\

\end{proof}

\begin{rem}\label{rem-ajut}{\rm
In computations we shall often use
$$
{\mathcal W}(\hat{f}) = \left[ \frac{\partial f}{\partial
\overline{t_{0}^{\beta}}}\right]^{\wedge}
\overline{\eta^{\beta\sigma} << \tau_{0,1}\tau_{0,\sigma},
\overline{\mathcal W} >>_{{}_0}},
$$
for any $f\in C^{\infty}(M)$ and
$\mathcal W\in {\mathcal T}^{0,1}_{M^{\infty}}$,
and also
\begin{equation}\label{tu-conj}
\overline{T({\mathcal W})} (\hat{f})=0, \quad \forall f\in
C^{\infty}(M), \quad \forall {\mathcal W}\in {\mathcal
T}^{1,0}_{M^{\infty}}.
\end{equation}
These follow by taking conjugations of (\ref{x-f}) and
(\ref{hat-kill}), and using that conjugations commute with natural
lifts of functions.}

\end{rem}

\subsection{Natural lifts of arbitrary tensor fields}\label{general}

Using the above ideas  one may
lift any tensor field, from $M$ to $M^{\infty}$, in the
following way. We first extend componentwise the endomorphism $T$, defined by
(\ref{endom-T}), to products $T^{1,0}M^{\infty}\times \cdots
\times T^{1,0}M^{\infty}$ ($p\geq 0$ factors; when $p=0$ the product
reduces to the trivial bundle $M^{\infty}\times\mathbb{C}$ and
$T$ is the identity operator).
Similarly, we define a map
\begin{equation}\label{ridic}
{}^{\wedge} : {\mathcal T}^{1,0}M\times \cdots \times {\mathcal T}^{1,0}M
\rightarrow
 {\mathcal T}^{1,0}M^{\infty}\times \cdots \times {\mathcal T}^{1,0}
M^{\infty} ,
\end{equation}
($p$-factors in both products) which, for $p\geq 1$, is given by
$$
(X_{1},
\cdots , X_{p})^{\wedge}:= (\widehat{X_{1}}, \cdots  , \widehat{X_{p}}).
$$
For $p=0$, the map (\ref{ridic})  is defined on  $C^{\infty}(M)$ and
is just the natural lift of functions.
With this preliminary notation,
let
${\mathcal F}$ be a $(p,q)$-tensor field  on $M$,
i.e. a map
$$
{\mathcal F}: T^{1,0}M\times \cdots \times T^{1,0}M \rightarrow
T^{1,0}M\times \cdots \times T^{1,0}M,
$$
($p$-factors in the left hand side, $q$  in the right hand side), which is
$C^{\infty}(M)$-linear  (in all arguments) or $C^{\infty}(M, \mathbb{R})$-linear and
complex anti-linear (in all arguments).
The natural lift
$\hat{\mathcal F}$ of $\mathcal F$ is a tensor field of the same type on $M^{\infty}$, defined by
$$
\hat{\mathcal F}\left( T^{n_{1}}(\tau_{0,\alpha_{1}}), \cdots ,
T^{n_{p}}(\tau_{0,\alpha_{p}})\right)=
 \delta_{n_{1},\cdots ,n_{p}} T^{n_{i}} \left(\left[ {\mathcal
F}(\tau_{0,\alpha_{1}}, \cdots ,
\tau_{0,\alpha_{p}})\right]^{\wedge}\right) ,
$$
where $\delta_{n_{1},\cdots ,n_{p}}=0$ unless all $n_{i}$ are
equal and $\delta_{n\cdots n}=1$.

\begin{rem}{\rm From the definition of $\hat{\mathcal F}$, for
any vector fields
$X_{1}, \cdots , X_{p}\in {\mathcal T}^{1,0}_{M}$,
\begin{equation}\label{more-general}
\hat{\mathcal F}\left( T^{n_{1}}(\hat{X}_{1}), \cdots ,
T^{n_{p}}(\hat{X}_{p})\right) = \delta_{n_{1},\cdots ,n_{p}}
T^{n_{i}}\left( \left[ {\mathcal F}(X_{1}, \cdots ,
X_{p})\right]^{\wedge}\right) .
\end{equation}}
\end{rem}

\begin{ex}{\rm i) Liu's metric (\ref{hat-begin}) (see \cite{liu4}) already mentioned in the Introduction is the natural
lift of the Poincar\'e metric ${\eta}$:
\begin{equation}\label{hat-eta}
\hat{\eta}\left(T^{n}(\tau_{0,\alpha}),
T^{m}(\tau_{0,\beta})\right) =\delta_{mn} \eta (
\tau_{0,\alpha},\tau_{0,\beta}),\quad\forall m,n
\end{equation}
the right hand side of the above expression being constant, hence
coincides with its natural extension.\\

ii) The natural lift $\hat{A}$ of an endomorphism $A$ of
$T^{1,0}M$ (viewed as a $(1,1)$-tensor field) is given by
$$
\hat{A} \left( T^{m}(\tau_{0,\alpha})\right) = T^{m} \left(
\left[A(\tau_{0,\alpha})\right]^{\wedge}\right) .
$$
Note that $[A,B]^{\wedge} = [\hat{A},\hat{B}]$, for any  endomorphisms $A$ and $B$ of $T^{1,0}M.$
}
\end{ex}

\section{The lifted pseudo-Hermitian
metric}\label{hermitian-str}

We now consider a real structure $k:T^{1,0}M\rightarrow T^{1,0}M$
compatible with $\eta$, i.e. $h:= \eta (\cdot , k\cdot )$ is a
pseudo-Hermitian metric. It is easy to check that its natural lift
$\hat{k}:T^{1,0}M^{\infty}\rightarrow T^{1,0}M^{\infty}$
is a real structure, compatible with the natural lift $\hat{\eta}$
of $\eta$, and that
$\hat{k}$ and $\hat{\eta}$  give rise to a
pseudo-Hermitian metric $\hat{h}=\hat{\eta}(\cdot , \hat{k}\cdot)$,
which is the natural lift
of ${h}.$ In this section we compute the Chern connection and the
curvature of $\hat{h}.$ For completeness of our exposition, we recall
the expression of $\hat{h}.$

\begin{defn}\label{h-k-def} The natural lift $\hat{h}$ of $h$ to $M^{\infty}$
is defined by
\begin{equation}\label{hat-h-0}
\hat{h}\left(T^{n}(\tau_{0,\alpha}), T^{m}(\tau_{0,\beta})\right) =\delta_{mn}\left[ h( \tau_{0,\alpha},\tau_{0,\beta})\right]^{\wedge},\quad\forall
m,n.
\end{equation}

\end{defn}

\noindent More generally,
\begin{equation}\label{hat-h}
\hat{h}( T^{n}(\hat{X}), T^{m}(\hat{Y}))= \delta_{mn}
\left[ h(X,Y)\right]^{\wedge}, \quad\forall X,Y\in {\mathcal T}^{1,0}_{M}.
\end{equation}

To simplify the expression of the Chern connection of $\hat{h}$
and other expressions we define the functions (following Getzler \cite{Getzler})
\begin{equation}\label{M}
M^{\gamma}_{\alpha}:= \eta^{\gamma\sigma} <<
\tau_{0,1}\tau_{0,\sigma}\tau_{0,\alpha}>>_{{}_0},\quad 1\leq \alpha
,\gamma \leq N
\end{equation}
and study some of their basic properties (see Remark \ref{deriv-rem}
and Lemma \ref{deriv-M}).

\begin{rem}\label{deriv-rem}{\rm With the above notation,
relation (\ref{x-f}) implies
\begin{equation}\label{deriv-new}
\tau_{0,\alpha} (\hat{f}) = \left[\frac{\partial f}{\partial
t_{0}^{\beta}} \right]^{\wedge}M_{\alpha}^{\beta},\quad \forall f\in
C^{\infty}(M).
\end{equation}
Notice that relation (\ref{deriv-new}) implies
$M_{\alpha}^{\sigma} = \tau_{0,\alpha}(\widehat{f^{\sigma}})$,
where $f^{\sigma}(t_{0}^{1},\cdots , t_{0}^{N})= t_{0}^{\sigma}.$
It also follows from the String Equation that
\[
\left. M^\sigma_\alpha \right|_M = \delta^\sigma_\alpha\,.
\]
}
\end{rem}

In what follows we shall use the derivatives of the functions $M_{\alpha}^{\sigma}$
along vector fields from the image of the operator $T$.
We will compute these derivatives now, and then compute the Chern connection of $\hat{h}.$

\begin{lem}\label{deriv-M}
The following relations hold: for any $1\leq \alpha ,\beta
,\sigma\leq N$,
$$
T(\tau_{0,\beta})(M_{\alpha}^{\sigma})= (\tau_{0,\alpha}\circ
\tau_{0,\beta})(\widehat{f^{\sigma}})
= \eta^{\sigma \mu}<< \tau_{0,1}\tau_{0,\mu}(\tau_{0,\alpha}\circ
\tau_{0,\beta})>>_{{}_0}
$$
and
$$
T^{q}(\tau_{0,\beta})(M_{\alpha}^{\sigma})=0,\quad\forall q\geq 2.
$$
\end{lem}

\begin{proof}
We use $M_{\alpha}^{\sigma} =
\tau_{0,\alpha}(\widehat{f^{\sigma}})$ (from the previous remark)
and that $\widehat{f^{\sigma}}$ is annihilated by
$T^{q}(\tau_{0,\beta})$ for any $q\geq 1$ (see Lemma \ref{ian}).
Recall, also, that
$$
[T^{q}(\tau_{0,\beta}), \tau_{0,\alpha}] =
T^{q-1}(\tau_{0,\alpha}\circ \tau_{0,\beta}),
$$
(see Lemma \ref{cr}). We obtain: for any $q\geq 1$,
\begin{eqnarray*}
T^{q}(\tau_{0,\beta})\left( M_{\alpha}^{\sigma}\right) & = & T^{q}
(\tau_{0,\beta})\tau_{0,\alpha}(\widehat{f^{\sigma}})\\
& = & [T^{q}(\tau_{0,\beta}), \tau_{0,\alpha} ] ( \widehat{f^{\sigma}})
 = T^{q-1}\left( \tau_{0,\alpha}\circ\tau_{0,\beta}\right)(
\widehat{f^{\sigma}}) ,
\end{eqnarray*}
which vanishes when $q\geq 2$. When $q=1$, we obtain
$$
T(\tau_{0,\beta})\left( M_{\alpha}^{\sigma}\right) =
(\tau_{0,\alpha}\circ\tau_{0,\beta}) ( \widehat{f^{\sigma}}) =
\eta^{\sigma \mu} <<\tau_{0,1}\tau_{0,\mu}(\tau_{0,\alpha}\circ
\tau_{0,\beta})>>_{{}_0}
$$
where in the last equality  we used
(\ref{x-f}), with $f:= f^{\sigma}$ and
${\mathcal W}=\tau_{0,\alpha}\circ\tau_{0,\beta}$.
\end{proof}

From the above lemma, $T^q(\tau_{0,\beta})(M^\sigma_\alpha)$ is
symmetric in $\alpha$ and $\beta\,.$

\begin{lem}\label{chern}
Let $D$ be the Chern connection of $h.$ The Chern connection $\hat{D}$ of $\hat{h}$ is given by: for any $n\geq 0$ and $1\leq \alpha , \beta\leq N$,
\begin{equation}\label{chern-1}
\hat{D}_{\tau_{0,\alpha}}\left( T^{n}(\tau_{0,\beta})\right) =
M^{\sigma}_{\alpha} T^{n}\left( \left[
D_{\tau_{0,\sigma}}(\tau_{0,\beta})\right]^{\wedge}\right)
\end{equation}
and
\begin{equation}\label{chern-2}
\hat{D}_{T({\mathcal W})}\left( T^{n}( \tau_{0,\alpha})\right) =0, \quad\forall {\mathcal W}\in{\mathcal T}^{1,0}_{M^{\infty}}.
\end{equation}

\end{lem}

\begin{proof}
By definition, the Chern connection
$\hat{D}$ satisfies
\begin{equation}\label{chern-det}
{\mathcal W}_{1}\hat{h}({\mathcal W}_{2}, {\mathcal W}_{3}) =
\hat{h}( \hat{D}_{{\mathcal W}_{1}}\mathcal W_{2}, \mathcal
W_{3}),\quad \forall {\mathcal W}_{1}, {\mathcal W}_{2}, {\mathcal
W}_{3}\in {\mathcal T}_{M^{\infty}}.
\end{equation}
Let ${\mathcal W}_{1}:= \tau_{0,\alpha}$, ${\mathcal W}_{2}:=
T^{n}(\tau_{0,\beta})$, ${\mathcal W}_{3}:=
T^{m}(\tau_{0,\gamma})$. With these arguments, equation
(\ref{chern-det}) becomes
\begin{equation}\label{delta-m-n}
\delta_{nm}\tau_{0,\alpha }(\widehat{h_{\beta \gamma}}) = \hat{h}
\left( \hat{D}_{\tau_{0,\alpha}} \left(
T^{n}(\tau_{0,\beta})\right) , T^{m} (\tau_{0,\gamma })\right) ,
\end{equation}
where,  to simplify notation, we  defined $ h_{\beta \gamma} :=h(\tau_{0,\beta
}, \tau_{0,\gamma }).$ But from (\ref{deriv-new}),
\begin{equation}\label{delta-m}
\tau_{0,\alpha}( \widehat{h_{\beta \gamma}})  = \left[
\frac{\partial h_{\beta\gamma}}{\partial
t_{0}^{\sigma}}\right]^{\wedge}
M_{\alpha}^{\sigma} = M^{\sigma}_{\alpha}
\left[ h\left(  D_{\tau_{0,\sigma}}(\tau_{0,\beta}),
\tau_{0,\gamma}\right) \right]^{\wedge}.
\end{equation}
Combining (\ref{delta-m-n}) with (\ref{delta-m}) we obtain:
$\forall n,m,\alpha ,\beta,$
$$
M_{\alpha}^{r} \hat{h} \left( T^{n}\left( \left[
D_{\tau_{0,r}}(\tau_{0,\beta})\right]^{\wedge}\right) ,
T^{m}(\tau_{0,\gamma })\right)\\
= \hat{h} \left( \hat{D}_{\tau_{0,\alpha}}\left(
T^{n}(\tau_{0,\beta}) \right) , T^{m} (\tau_{0,\gamma })\right) .
$$
From the non-degeneracy of $\hat{h}$, we obtain (\ref{chern-1}).

We now  prove (\ref{chern-2}). With ${\mathcal W}_{2}$ and
${\mathcal W}_{3}$ as above, $\hat{h}({\mathcal W}_{2}, {\mathcal
W}_{3})$ is the natural extension of $\delta_{mn}
h_{\beta\gamma}$. Thus, if ${\mathcal W}_{1}= T({\mathcal W})$ in
the image of $T$,  it annihilates $\hat{h}({\mathcal W}_{2},
{\mathcal W}_{3})$ (see Lemma \ref{ian}) and from
(\ref{chern-det}),
\begin{equation}\label{ch1}
\hat{h}\left( \hat{D}_{T({\mathcal W})}\left(  T^{n} (\tau_{0,\beta })\right), T^{m}(\tau_{0, \gamma })\right) = 0.
\end{equation}
Using the non-degeneracy of $\hat{h}$ again,
we obtain (\ref{chern-2}).

\end{proof}

From Lemmas \ref{ian} and \ref{chern},
the covariant derivatives with respect to $\hat{D}$
of vector fields $T^{n}(\hat{X})$, in directions
from the image of $T$,  vanish:
\begin{equation}\label{d-d}
\hat{D}_{T({\mathcal W})} \left( T^{n}(\hat{X})\right) =0,\quad
\forall {\mathcal W}\in {\mathcal T}_{M^{\infty}},\quad X\in
{\mathcal T}_{M}.
\end{equation}
This fact will be used often in the computations from the next
sections.

\begin{lem}\label{curvature}The curvature
${}^{\hat{D}}\! R$ of $\hat{D}$ is related to the curvature
${}^{D} \! R$ of $D$ by:
\begin{equation}\label{curv1}
{}^{\hat{D}}\! R_{\tau_{0,\alpha},
\overline{\tau_{0,\beta}}}\left(  T^{r} (\tau_{0,\gamma} )\right)
=M_{\alpha}^{\sigma} \overline{M^{\nu}_{\beta}} T^{r}\left( \left[
{}^{D}\!
R_{\tau_{0,\sigma},\overline{\tau_{0,\nu}}}(\tau_{0,\gamma})\right]^{\wedge}\right)
,
\end{equation}
where  $1\leq\alpha , \beta ,\gamma\leq N$ and $r\geq 0$. If $m$
or $n$ is bigger than zero,
$$
{}^{\hat{D}}\! R_{T^{n}(\tau_{0,\alpha}),
\overline{T^{m}(\tau_{0,\beta})}} \left( T^{r} (\tau_{0,\gamma}
)\right) =0.
$$

\end{lem}

\begin{proof}
For any ${\mathcal W}_{1}, {\mathcal W}_{2}, {\mathcal W}_{3}\in {\mathcal T}_{M^{\infty}}$,
\begin{eqnarray*}
{}^{\hat{D}}\! R_{{\mathcal W}_{1}, \bar{\mathcal W}_{2}}({\mathcal W}_{3}) & = &
\hat{D}_{{\mathcal W}_{1}} \hat{D}_{\bar{\mathcal W}_{2}} ({\mathcal W}_{3}) - \hat{D}_{\bar{\mathcal W}_{2}}\hat{D}_{\mathcal W_{1}} ({\mathcal W}_{3})
-\hat{D}_{[{\mathcal W}_{1},\bar{\mathcal W}_{2}]}({\mathcal W}_{3})\\
& = &  -\bar{\partial}_{\bar{\mathcal W}_{2}}\hat{D}_{\mathcal W_{1}}(\mathcal W_{3}),
\end{eqnarray*}
where we used
$$
[{\mathcal W}_{1}, \bar{\mathcal W}_{2}]=0,\quad \hat{D}_{\bar{\mathcal W}_{2}}({\mathcal W}_{3}) = \bar{\partial}_{\bar{\mathcal W}_{2}}({\mathcal W}_{3}) = 0
$$
because $\mathcal W_{i}$ are holomorphic and
$\hat{D}$ is a Chern connection (hence its $(0,1)$-part is the $\bar{\partial}$ operator). Let
$$
{\mathcal W}_{1}:= T^{n}(\tau_{0,\alpha}),\quad {\mathcal W}_{2}:= T^{m}(\tau_{0,\beta}), \quad {\mathcal W}_{3}:=
T^{r}(\tau_{0,\gamma}).
$$
With these arguments,
\begin{equation}
{}^{\hat{D}}\! R_{ T^{n}(\tau_{0,\alpha}), \overline{
T^{m}(\tau_{0,\beta})}} (T^{r}(\tau_{0,\gamma})) = -
\bar{\partial}_{\overline{T^{m}(\tau_{0,\beta})}
}\hat{D}_{T^{n}(\tau_{0,\alpha})} \left(
T^{r}(\tau_{0,\gamma})\right) .
\end{equation}
From Lemma \ref{chern},  if $n>0$,
$$
{}^{\hat{D}} \! R_{ T^{n}(\tau_{0,\alpha}), \overline{
T^{m}(\tau_{0,\beta})}} (T^{r}(\tau_{0,\gamma}) )=0.
$$
If $n=0$, then, again from Lemma \ref{chern},
\begin{eqnarray}
\nonumber {}^{\hat{D}}\! R_{\tau_{0,\alpha},
\overline{T^{m}(\tau_{0,\beta})}}\left(
T^{r}(\tau_{0,\gamma})\right)
& = & -\overline{\partial}_{\overline{T^{m}(\tau_{0,\beta})}}\left( M^{\sigma}_{\alpha}
T^{r} \left( \left[ D_{\tau_{0,\sigma}}(\tau_{0,\gamma})\right]^{\wedge}\right) \right)\\
\label{ajut}
& = & - M^{\sigma}_{\alpha} T^{r} \left(
\bar{\partial}_{\overline{ T^{m}(\tau_{0,\beta})}}\left( \left[
D_{\tau_{0,\sigma}}(\tau_{0,\gamma})\right]^{\wedge}\right)
\right)
\end{eqnarray}
where in the last relation we used that $T$ and
$M^{\sigma}_{\alpha}$ are holomorphic.
We  define functions $f_{\sigma\gamma}^{\mu}$ on $M$ by the formula
\begin{equation}\label{def-f}
D_{\tau_{0,\sigma}}(\tau_{0,\gamma} )= f_{\sigma\gamma}^{\mu} \tau_{0,\mu}\,.
\end{equation}
With this
notation, relation (\ref{ajut}) becomes
\begin{equation}\label{curv-m}
{}^{\hat{D}}\! R_{\tau_{0,\alpha},
\overline{T^{m}(\tau_{0,\beta})}}\left(
T^{r}(\tau_{0,\gamma})\right) = - \left\{M^{\sigma}_{\alpha}  \overline{
T^{m}(\tau_{0,\beta})}\left(\widehat{f^{\mu}_{\sigma\gamma}}
\right)\right\} \, T^{r}(\tau_{0,\mu})
.
\end{equation}

We now distinguish two cases:\\

1) If $m=0$, then, from Remark
\ref{rem-ajut},
$$
\overline{{\tau}_{0,\beta}}\left(\widehat{f_{\sigma\gamma}^{\mu}}\right)
= \left[ \frac{\partial f_{\sigma\gamma}^{\mu}}{\partial
\overline{t_{0}^{\nu}}} \right]^{\wedge} \overline{M_{\beta}^{\nu}}
$$
and
\begin{eqnarray*}
{}^{\hat{D}}\! R_{\tau_{0,\alpha},
\overline{\tau_{0,\beta}}}\left( T^{r}(\tau_{0,\gamma})\right) & = &  -
M^{\sigma}_{\alpha} \overline{M^{\nu}_{\beta}}  \left[
\frac{\partial f^{\mu}_{\sigma\gamma}}{\partial
\overline{t_{0}^{\nu}}}\right]^{\wedge} T^{r}(\tau_{0,\mu})\\
& = &
M^{\sigma}_{\alpha}\overline{ M^{\nu}_{\beta}} T^{r} \left( \left[
{}^{D}\! R_{\tau_{0,\sigma}, \overline{\tau_{0,\nu}}}
(\tau_{0,\gamma})\right]^{\wedge}\right)\,,
\end{eqnarray*}
which implies  (\ref{curv1}).

2) If $m>0$, the right hand side of (\ref{curv-m}) is zero, because
$\widehat{f^{s}_{\sigma\gamma}}$ is annihilated by
$\overline{T^{m} (\tau_{0,\beta})}$ (again from Remark
\ref{rem-ajut}).

\end{proof}

One of the fundamental properties of the pseudo-Hermitian metric
$h$ is that it must be compatible with the holomorphic metric
$\eta\,,$ namely $D\eta =0\,$ \cite{cecotti,dubrovin3}\,. It will
turn out that the lifted metrics satisfy this same condition on
the big phase space (see Section \ref{extended-sect}).

\section{The lifted Higgs field}\label{higgs-str}

In this section we lift the
Higgs field from the small phase space to
the big phase space. Unlike other types of tensor fields
involved in our constructions,
the Higgs field on the big phase space is not obtained
via the general lifting procedure developed in Section
\ref{general}. One motivation for its definition is explained
in Remark \ref{ian-notes}.

\begin{defn}\label{def-hat-c}
Let $C_{X}Y = X\bullet Y$ be the Higgs field on $M$.
The Higgs field $\hat{C}$ on
$M^{\infty}$ is defined by
\begin{equation}\label{higgs-hat-1}
\hat{C}_{\tau_{0,\alpha} }\left( T^{n} (\tau_{0,\beta})\right) =
M^{\sigma}_{\alpha}T^{n}\left( \left[ C_{\tau_{0,\sigma}}
\left(  \tau_{0,\beta}\right) \right]^{\wedge}\right)
\end{equation}
and
\begin{equation}
\hat{C}_{T^{m}(\tau_{0,\alpha})}\left(
T^{n}(\tau_{0,\beta})\right) =0,
\end{equation}
for any  $n\geq 0$, $m\geq 1$ and $1\leq \alpha , \beta \leq N.$

\end{defn}

We make some comments on the above definition.

\begin{rem}\label{ian-notes}{\rm
The Higgs field $\hat{C}$, restricted to primary vector fields,
gives the quantum product. This fact relies on a property proved
in \cite{DijkgraafWitten}, namely that $2$-point functions
$<<\tau_{0,\alpha}\tau_{0,\beta}>>_{{}_0}$  on $M^{\infty}$ are
natural lifts of their restrictions to $M$. More precisely,
to prove that $\hat{C}_{\tau_{0,\alpha}}(\tau_{0,\beta})=
\tau_{0,\alpha}\circ\tau_{0,\beta}$, we notice that
\begin{align*}
&<<\tau_{0,\alpha}\tau_{0,\beta}\tau_{0,\gamma}>>_{{}_0}  = \tau_{0,\gamma}
\left( <<\tau_{0,\alpha}\tau_{0,\beta}>>_{{}_0}\right)\\
&= \tau_{0,\gamma}\left( <<\tau_{0,\alpha}\tau_{0,\beta}>>_{0}\vert_{M}\right)^{\wedge}
 =  \left( \frac{\partial}{\partial t_{0}^{\sigma}}<< \tau_{0,\alpha}
\tau_{0,\beta}>>_{{}_0}\vert_{M}\right)^{\wedge}
M^{\sigma}_{\gamma},\\
\end{align*}
for any $1\leq \alpha , \beta ,\gamma\leq N$. We obtain:
\begin{equation}\label{3-point}
<<\tau_{0,\alpha}\tau_{0,\beta}\tau_{0,\gamma}>>_{{}_0}=
\left( << \tau_{0,\alpha}
\tau_{0,\beta}\tau_{0,\sigma}>>_{{}_0}\vert_{M}\right)^{\wedge}
M^{\sigma}_{\gamma} .
\end{equation}
In particular, the right hand side of (\ref{3-point}) is symmetric
in $\alpha$ and $\gamma $ and thus
$$
<< \tau_{0,\alpha}
\tau_{0,\beta}\tau_{0,\gamma}>>_{{}_0} = \left( << \tau_{0,\gamma}
\tau_{0,\beta}\tau_{0,\sigma}>>_{{}_0}\vert_{M}\right)^{\wedge}
M^{\sigma}_{\alpha}.
$$
From the above relation we then obtain
\begin{align*}
&\tau_{0,\alpha}\circ\tau_{0,\beta}=
<< \tau_{0,\alpha}\tau_{0,\beta}\tau_{0,\gamma} >>_{{}_0} \eta^{\gamma\nu}
\tau_{0,\nu}\\
& = M^{\sigma}_{\alpha} \left(
<<\tau_{0,\gamma}\tau_{0,\beta}\tau_{0,\sigma}>>_{{}_0}\vert_{M}
\eta^{\gamma\nu}\tau_{0,\nu}\right)^{\wedge}\\
&= M^{\sigma}_{\alpha}\left[
C_{\tau_{0,\sigma}}(\tau_{0,\beta})\right]^{\wedge} =
\hat{C}_{\tau_{0,\alpha}}(\tau_{0,\beta}) ,
\end{align*}
as claimed. This is represented in Figure 1.

\setlength{\unitlength}{0.8mm}
\begin{picture}(100,120)
\qbezier(10,10)(15,50)(50,100)

\put(10,10){\line(4,3){20}}
\put(10,10){\line(1,0){60}}
\put(30,25){\line(1,0){60}}
\put(70,10){\line(4,3){20}}

\put(37,80){\line(4,3){20}}
\put(37,80){\line(1,0){60}}
\put(57,95){\line(1,0){60}}
\put(97,80){\line(4,3){20}}

\put(5,5){$M$}
\put(40,100){$M^\infty$}
\put(90,20){$T_pM$}
\put(120,90){$T_pM^\infty$}

\put(45,17){\vector(-1,0){7}}
\put(32,16){$X$}

\put(72,87){\vector(-1,0){7}}
\put(59,86){$X$}

\put(45,17){\vector(1,1){5}}
\put(51,20){$Y$}

\put(45,30){\vector(0,1){45}}

\put(72,87){\vector(1,1){5}}
\put(78,90){$Y$}

\put(45,17){\vector(4,-1){10}}
\put(55,12){$X\bullet Y$}

\put(72,87){\vector(4,-1){10}}
\put(82,82){$X\circ Y$}

\put(50,50){$\tau_{0,\alpha}\circ\tau_{0,\beta}=M^{\sigma}_{\alpha}\left[
{\tau_{0,\sigma}}\bullet\tau_{0,\beta}\right]^{\wedge}$}

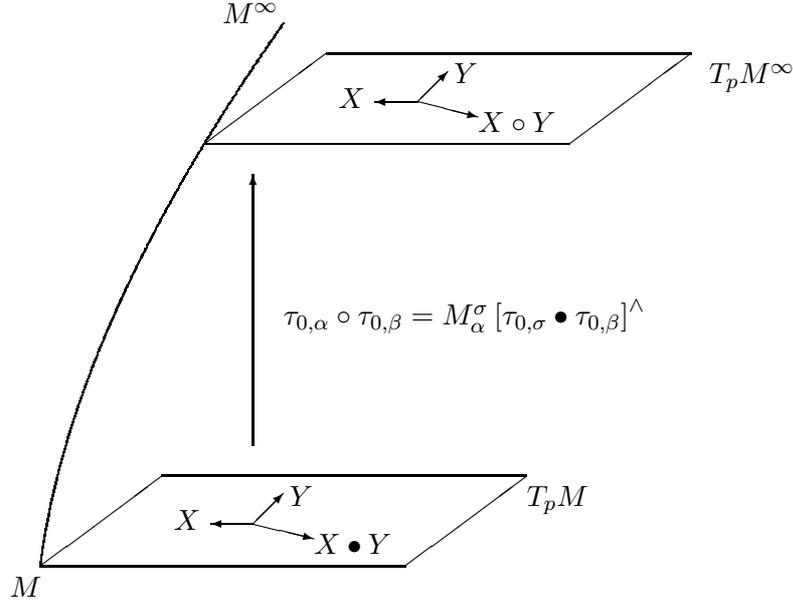
\put(20,-5){Figure~1:~The~lifting~of~the~multiplication~from~$M$~to~$M^\infty\,.$}

\end{picture}

}

\end{rem}

\bigskip

\bigskip

\begin{rem}{\rm Note also that $\hat{C}$ mirrors
the properties of the Chern connection $\hat{D}$ given in Lemma
\ref{chern}. Thus, each term in the second $tt^{*}$-equation
\[
{}^{\hat{D}} \! R_{Z_{1},\bar{Z}_{2}} + [\hat{C}_{Z_{1}},
\hat{C}^{\dagger}_{\bar{Z}_{2}}]=0,
\]
behaves in the same manner. Owing to this, the second
$tt^{*}$-equation is inherited from $M$ to $M^{\infty}$ (we shall
give details of this fact in Section \ref{harmonic-sect}). This
would not occur if Liu's multiplication $\diamond$ or the quantum
product $\circ$ were used, instead of $\hat{C}$.}
\end{rem}

\begin{lem} The metric $\hat{\eta}$ is invariant with respect to
$\hat{C}$, i.e.
\begin{equation}\label{verificare}
\hat{\eta} (\hat{C}_{\mathcal W_{1}} (\mathcal W_{2}), {\mathcal
W}_{3})=\hat{\eta} ({\mathcal W}_{2} , \hat{C}_{\mathcal
W_{1}}({\mathcal W}_{3})),
\end{equation}
for any vector fields ${\mathcal W}_{1}, {\mathcal W}_{2},
{\mathcal W}_{3}\in {\mathcal T}^{1,0}_{M^{\infty}}.$
\end{lem}

\begin{proof} From
the definition of $\hat{C}$ both sides of (\ref{verificare}) are
zero, when ${\mathcal W}_{1}\in \mathrm{Im}(T).$ When $\mathcal
W_{1}= \tau_{0,\alpha}$,
$$
\hat{\eta} (\hat{C}_{\tau_{0,\alpha}} T^{n}(\tau_{0,\beta}),
T^{m}(\tau_{0,\gamma})) = \delta_{nm} M^{\sigma}_{\alpha} \left[
\eta\left( C_{\tau_{0,\sigma}}(\tau_{0,\beta}),
\tau_{0,\gamma}\right)\right]^{\wedge}
$$
which is symmetric in the pairs $(n,\beta )$ and $(m,\gamma )$, because
$\eta$ is symmetric and  invariant with respect to $C$.
\end{proof}

\section{The lifted $tt^{*}$-bundles}\label{extended-str}

In this section we consider various compatibility conditions (see
Section \ref{frob})  on the structures on the small phase
space and we show that they are inherited by the lifted
structures on the big phase space. We preserve the notation from
the previous sections. Thus, the small phase space $M$ has a
metric $\eta$ (the Poincar\'e pairing), a multiplication $\bullet$ (with
Higgs field $C$) and a real structure $k$ which is compatible with $\eta$.
Recall that $k$ compatible with $\eta$ means that $h:= \eta (\cdot ,
k\cdot )$ is a pseudo-Hermitian metric. We denote by $D$ the Chern
connection of $h$.
Since $M$ is a Frobenius manifold, $(\nabla , \eta , C,
R_{0}, R_{\infty})$ is a Saito structure on $T^{1,0}M$, where
$\nabla$ is the Levi-Civita connection of $\eta$ (which is flat
and $\nabla (\tau_{0,\alpha})=0$, for any $1\leq \alpha \leq N$),
$R_{0}=C_{E}$ is the multiplication by the Euler field
and $R_{\infty}=\nabla E.$

The lifted structures were defined in the previous sections and
will be denoted as before, $\hat{\eta}$, $\hat{C}$, $\hat{k}$ and
$\hat{h}$ (with Chern connection $\hat{D}$). We shall also
consider the natural lifts $\hat{R}_{0}$ and
$\hat{R}_{\infty}$, which are endomorphisms on
$T^{1,0}M^{\infty}.$  Finally, we
need to define a flat connection $\hat{\nabla}$ on the bundle
$T^{1,0}M^{\infty}$. The connection $\hat{\nabla}$ will be part of
a Saito structure on the big phase space and is defined as
follows.

\begin{defn} The connection $\hat{\nabla}$ on $T^{1,0}M^{\infty}$
is defined by the condition that all vector fields
$T^{m}(\tau_{0,\alpha })$, ($m\geq 0,\alpha = 1,\cdots , N$) are
$\hat{\nabla}$-parallel.
\end{defn}

Note that the $\hat{\nabla}$-covariant derivatives of vector
fields $T^{m}(\hat{X})$, in directions from the image of $T$,
vanish, i.e.
\begin{equation}\label{nabla-t}
\hat{\nabla}_{T({\mathcal W})}\left(T^{m}(\hat{X})\right) =0,\quad
\forall\mathcal W\in {\mathcal T}_{M^{\infty}},\quad X\in
{\mathcal T}_{M}, \quad m\geq 0.
\end{equation}
This follows from the above definition and Lemma \ref{ian}.
Note also that the two connections
$\hat{\nabla}$ and $\nabla$
(the latter defined in Lemma \ref{cr})  are connected by the relation
\[
\left(
{\hat{\nabla}}_{T^n(\tau_{0,\alpha})} -
{{\nabla}}_{T^n(\tau_{0,\alpha})}\right)\left( T^m(\tau_{0,\beta})\right)
= \delta_{n,0} T^{m-1} ( \tau_{0,\alpha} \circ \tau_{0,\beta}) ,
\]
though we will not have to use this result.

Since the various notions from $tt^{*}$-geometry are built as an
enrichment of the notion of Saito bundle, we begin by proving that
the lifted data provides a Saito structure on the big phase
space.

\subsection{Lifted Saito bundles}\label{subsect-saito}

\begin{thm}\label{thm-saito} The data $(\hat{\nabla}, \hat{\eta}, \hat{C},
\hat{R}_{0}, \hat{R}_{\infty})$ is a Saito structure on
$T^{1,0}M^{\infty}$.
\end{thm}

To prove Theorem \ref{thm-saito}, we need to check that
$(\hat{\nabla}, \hat{\eta}, \hat{C},\hat{R}_{0}, \hat{R}_{\infty})$
satisfies all defining conditions for Saito structures
(see Definition \ref{saito}). In particular, in Lemma \ref{lem-1} we show that the potentiality condition $d^{\hat{\nabla}}\hat{C}=0$
holds and in Lemma \ref{lem-2} we show that the relation
$$
\hat{\nabla}(\hat{R}_{0})+\hat{C} = [\hat{C}, \hat{R}_{\infty}]
$$
holds. The remaining conditions
from Definition \ref{saito} can be checked easily,
and we omit the proofs.

\begin{lem}\label{lem-1}
The  equality $d^{\hat{\nabla}}\hat{C}=0$ holds.
\end{lem}

\begin{proof}
From the various definitions it follows that
\begin{align*}
&(d^{\hat{\nabla}} \hat{C})_{T^{p}(\tau_{0,\alpha}),
T^{q}(\tau_{0,\beta})} \left(T^{m}(\tau_{0,\gamma}) \right) \\&=
\hat{\nabla}_{T^{p}(\tau_{0,\alpha})}\left(
\hat{C}_{T^{q}(\tau_{0,\beta})}\right) \left(T^{m}(\tau_{0,\gamma})\right) - \hat{\nabla}_{T^{q}(\tau_{0,\beta})}\left(
\hat{C}_{T^{p}(\tau_{0,\alpha})}\right)
\left(T^{m}(\tau_{0,\gamma})\right)\\& \phantom{=}-
\hat{C}_{[T^{p}(\tau_{0,\alpha}),
T^{q}(\tau_{0,\beta})]}\left(T^{m}(\tau_{0,\gamma})\right) .\\
\end{align*}
When $p,q\geq 1$, both $\hat{C}_{T^{p}(\tau_{0,\gamma})}$ and
$\hat{C}_{T^{q}(\tau_{0,\beta})}$ are trivial, and from Lemma \ref{cr},
$[T^{p}(\tau_{0,\alpha}), T^{q}(\tau_{0,\beta})]=0\,.$ Hence

$$
(d^{\hat{\nabla}} \hat{C})_{T^{p}(\tau_{0,\alpha}),
T^{q}(\tau_{0,\beta})} \left(T^{m}(\tau_{0,\gamma})\right) = 0.
$$
When $p=q=0$ we obtain
\begin{align*}
&(d^{\hat{\nabla}}\hat{C})_{\tau_{0,\alpha},
\tau_{0,\beta}}T^{m}(\tau_{0,\gamma})\\
&=\hat{\nabla}_{\tau_{0,\alpha}} \left( \hat{C}_{\tau_{0,\beta}}
T^{m}(\tau_{0,\gamma})\right) -
\hat{\nabla}_{\tau_{0,\beta}} \left( \hat{C}_{\tau_{0,\alpha}}
T^{m}(\tau_{0,\gamma})\right)\,,\\
&= \hat{\nabla}_{\tau_{0,\alpha}}\left( M_{\beta}^{\nu} T^{m} \left[
C_{\tau_{0,\nu}}(\tau_{0,\gamma})\right]^{\wedge}\right) -
\hat{\nabla}_{\tau_{0,\beta}}\left( M_{\alpha}^{\nu} T^{m} \left[
C_{\tau_{0,\nu}}(\tau_{0,\gamma})\right]^{\wedge}\right)\,,\\
& = \phantom{-}\frac{\partial M_{\beta}^{\nu}}{\partial t_{0}^{\alpha}}
T^{m}\left(\left[
C_{\tau_{0,\nu}}(\tau_{0,\gamma})\right]^{\wedge}\right) +
M_{\beta}^{\nu} \hat{\nabla}_{\tau_{0,\alpha}}\left( T^{m}
\left[ C_{\tau_{0,\nu}}(\tau_{0,\gamma})\right]^{\wedge}\right) \,,\\
&\phantom{=} - \frac{\partial M_{\alpha}^{\nu}}{\partial t_{0}^{\beta}}
T^{m}\left(\left[
C_{\tau_{0,\nu}}(\tau_{0,\gamma})\right]^{\wedge}\right) -
M_{\alpha}^{\nu} \hat{\nabla}_{\tau_{0,\beta}}\left( T^{m}
\left[ C_{\tau_{0,\nu}}(\tau_{0,\gamma})\right]^{\wedge}\right) ,\\
\end{align*}
where we used the $\hat{\nabla}$-flatness of
$T^{m}(\tau_{0,\gamma})$ and the definition of $\hat{C}.$
From
$$
M_{\beta}^{\nu} = \eta^{\nu\mu} <<
\tau_{0,1}\tau_{0,\mu}\tau_{0,\beta}>>_{{}_0}
$$
and relation (\ref{deriv-tensor}),
\begin{equation}
\frac{\partial M_{\beta}^{\nu}}{\partial t_{0}^{\alpha}} = \eta^{\nu\mu}
<< \tau_{0,1}\tau_{0,\mu}\tau_{0,\beta}\tau_{0,\alpha}>>_{{}_0},
\end{equation}
which is symmetric in $\alpha$ and $\beta .$
Thus,
$$
(d^{\hat{\nabla}}\hat{C})_{\tau_{0,\alpha},
\tau_{0,\beta}}\left(T^{m}(\tau_{0,\gamma})\right)= M_{\beta}^{\nu}
\hat{\nabla}_{\tau_{0,\alpha}}\left( T^{m} \left[
C_{\tau_{0,\nu}}(\tau_{0,\gamma})\right]^{\wedge}\right) -
M_{\alpha}^{\nu} \hat{\nabla}_{\tau_{0,\beta}}\left( T^{m} \left[
C_{\tau_{0,\nu}}(\tau_{0,\gamma})\right]^{\wedge}\right) .
$$
Now, we write
\begin{equation}\label{constanta}
C_{\tau_{0,\beta}}(\tau_{0,\sigma}) =
c_{\beta\sigma}^{\mu}\tau_{0,\mu}
\end{equation}
where $c_{\beta\sigma}^{\mu}$ are the structure constants of the
Frobenius multiplication on $M$. With this notation, the above
relation becomes
\begin{equation}
(d^{\hat{\nabla}}\hat{C})_{\tau_{0,\alpha},
\tau_{0,\beta}}\left(T^{m}(\tau_{0,\gamma})\right) = \left(
M_{\beta}^{\mu}\tau_{0,\alpha}( \widehat{c_{\mu\gamma}^{\nu}}) -
M_{\alpha}^{\mu}\tau_{0,\beta}( \widehat{c_{\mu\gamma}^{\nu}})\right)
T^{m}(\tau_{0,\nu}).
\end{equation}
But from Lemma \ref{ian},
$$
\tau_{0,\alpha}( \widehat{c_{\mu\gamma}^{\nu}})=
\left[ \frac{\partial c_{\mu\gamma}^{\nu}}{\partial t_{0}^{\delta}}\right]^{\wedge}
M_{\alpha}^{\delta}
$$
and
$$
(d^{\hat{\nabla}}\hat{C})_{\tau_{0,\alpha},
\tau_{0,\beta}}T^{m}(\tau_{0,\gamma}) = \left[ \frac{\partial
c_{\mu\gamma}^{\nu}}{\partial t_{0}^{\delta}} \right]^{\wedge}\left(
M_{\alpha}^{\delta}M_{\beta}^{\mu} - M_{\beta}^{\delta}M_{\alpha}^{\mu}\right)
T^{m}(\tau_{0,\nu})
$$
which is zero, because
$\frac{\partial c_{\mu\gamma}^{\nu}}{\partial t_{0}^{\delta}}$
is symmetric in $\mu$ and $\delta$  (from
$d^{\nabla}C=0$) and $M_{\alpha}^{\delta}M_{\beta}^{\mu} - M_{\beta}^{\delta}M_{\alpha}^{\mu}$
is skew-symmetric in these indices (for any $\alpha$, $\beta$ fixed).

It remains to compute $(d^{\hat{\nabla}}
\hat{C})_{\tau_{0,\alpha}, T^{q}(\tau_{0,\beta})}
\left( T^{m}(\tau_{0,\gamma})\right)$, when $q\geq 1.$ Note that
\begin{align*}
&(d^{\hat{\nabla}} \hat{C})_{\tau_{0,\alpha},T^{q}(\tau_{0,\beta})}
\left( T^{m}(\tau_{0,\gamma})\right)\\
&=
-\hat{\nabla}_{T^{q}(\tau_{0,\beta})}\left(\hat{C}_{\tau_{0,\alpha}}
T^{m}(\tau_{0,\gamma})\right) -\hat{C}_{[\tau_{0,\alpha},
T^{q}(\tau_{0,\beta})]}T^{m}(\tau_{0,\gamma})\\
&= -\hat{\nabla}_{T^{q}
(\tau_{0,\beta})} \left( M_{\alpha}^{\sigma} T^{m}\left[
C_{\tau_{0,\sigma}}(\tau_{0,\gamma})\right]^{\wedge}\right)
+ \hat{C}_{T^{q-1}\left( \tau_{0,\alpha}\circ
\tau_{0,\beta}\right)} \left(
T^{m}(\tau_{0,\gamma})\right) \\
&= - T^{q} (\tau_{0,\beta}) (M_{\alpha}^{\sigma}) T^{m}\left(
\left[ C_{\tau_{0,\sigma}}(\tau_{0,\gamma})\right]^{\wedge}
\right)+ \hat{C}_{T^{q-1}\left( \tau_{0,\alpha}\circ
\tau_{0,\beta}\right)} \left(
T^{m}(\tau_{0,\gamma})\right) ,\\
\end{align*}
where we used $\hat{C}_{T^{q}(\tau_{0,\beta})}=0$,
$[ T^{q}(\tau_{0,\beta}), \tau_{0,\alpha}]= T^{q-1} (\tau_{0,\alpha}
\circ \tau_{0,\beta})$ and relation
(\ref{nabla-t}).
Thus,
\begin{align}
\nonumber (d^{\hat{\nabla}}
\hat{C})_{\tau_{0,\alpha},T^{q}(\tau_{0,\beta})} \left(
T^{m}(\tau_{0,\gamma})\right)& =
 - T^{q} (\tau_{0,\beta}) (M_{\alpha}^{\sigma}) T^{m}\left( \left[
C_{\tau_{0,\sigma}}(\tau_{0,\gamma})\right]^{\wedge} \right)\\
\label{D}& \phantom{=}+ \hat{C}_{T^{q-1}\left( \tau_{0,\alpha}\circ
\tau_{0,\beta}\right)} \left( T^{m}(\tau_{0,\gamma})\right) .
\end{align}
When $q\geq 2$, both terms from the right hand side of the above
equality vanish (for the first term, see Lemma \ref{deriv-M}). If
$q=1$, then again from Lemma \ref{deriv-M},
$$
T(\tau_{0,\beta}) (M_{\alpha}^{\sigma}) = (\tau_{0,\alpha}\circ
\tau_{0,\beta})(\widehat{f^{\sigma}})
$$
(recall that $f^{\sigma}(t_{0}^{1}, \cdots , t_{0}^{N}) =
t_{0}^{\sigma}$).  We obtain:
\begin{align*}
&(d^{\hat{\nabla}} \hat{C})_{\tau_{0,\alpha},T(\tau_{0,\beta})}
\left( T^{m}(\tau_{0,\gamma})\right)\\
&= -
\left(\tau_{0,\alpha}\circ\tau_{0,\beta}\right)(
\widehat{f^{\sigma}}) T^{m} \left(\left[ C_{\tau_{0,\sigma}}
(\tau_{0,\gamma})\right]^{\wedge}\right)  +
\hat{C}_{\tau_{0,\alpha}\circ \tau_{0,\beta}}
T^{m}(\tau_{0,\gamma}) .
\end{align*}
But the right hand side from the above
expression is zero, because, for any primary field $\mathcal W$ on
$M^{\infty}$ and $1\leq\gamma\leq N$,
\begin{equation}\label{int}
-{\mathcal W}( \widehat{f^{\sigma}}) T^{m}\left(
\left[C_{\tau_{0,\sigma}}(\tau_{0,\gamma})\right]^{\wedge}\right)
+\hat{C}_{\mathcal W}\left( T^{m}(\tau_{0,\gamma})\right) =0.
\end{equation}
(To prove (\ref{int}) we may assume, without loss of generality, that
${\mathcal W} =\tau_{0,\alpha}$; then we use  the definition of
$ \hat{C}_{\tau_{0,\alpha}}\left( T^{m}(\tau_{0,\gamma})\right)$
and $\tau_{0,\alpha}(\widehat{f^{\sigma}}) = M_{\alpha}^{\sigma}$, see Remark
\ref{deriv-rem}).  Our claim follows.

\end{proof}

\begin{lem}\label{lem-2}
The equality
\begin{equation}\label{need-1}
\hat{\nabla}_{\mathcal W} (\hat{R}_{0}) + \hat{C}_{\mathcal W }=
[\hat{C}_{\mathcal W}, \hat{R}_{\infty}],\quad\forall {\mathcal
W}\in {\mathcal T}_{M^{\infty}}
\end{equation}
holds.
\end{lem}

\begin{proof}
When ${\mathcal W}= T^{m}(\tau_{0,\alpha})$ with $m\geq 1$,
$\hat{C}_{\mathcal W}=0$ and relation (\ref{need-1}) is satisfied,
because
$$
\hat{\nabla}_{T^{m}(\tau_{0,\alpha})} (\hat{R}_{0})\left( T^{n}
(\tau_{0,\beta})\right) =
\hat{\nabla}_{T^{m}(\tau_{0,\alpha})}\left( T^{n}
\left[R_{0}(\tau_{0,\beta})\right]^{\wedge}\right) =0,
$$
where we used the definition
of $\hat{R}_{0}$ and relation
(\ref{nabla-t}).

It remains to check that
\begin{equation}\label{need-saito}
\hat{\nabla}_{\tau_{0,\alpha}} (\hat{R}_{0}) +
\hat{C}_{\tau_{0,\alpha} }= [\hat{C}_{\tau_{0,\alpha}},
\hat{R}_{\infty}].
\end{equation}

For this, recall that
\begin{equation}\label{extindere-cr-1}
\hat{C}_{\tau_{0,\alpha}}\left( T^{n}(\tau_{0,\beta})\right)
= M_{\alpha}^{\sigma}
T^{n}\left(\left[ C_{\tau_{0,\sigma}}(\tau_{0,\beta})\right]^{\wedge}
\right).
\end{equation}
Also, it is straightforward that
\begin{equation}\label{extindere-cr}
[\hat{C}_{\tau_{0,\alpha}}, \hat{R}_{\infty}] \left(
T^{n}(\tau_{0,\beta})\right) = M_{\alpha}^{\sigma} T^{n} \left(
[{C}_{\tau_{0,\sigma}},{R_{\infty}}](\tau_{0,\beta})^{\wedge}\right) .
\end{equation}
We now check that
\begin{equation}\label{required}
\hat{\nabla}_{\tau_{0,\alpha}}(\hat{R}_{0})\left(T^{n}(\tau_{0,\beta})\right)
= M_{\alpha}^{\sigma}T^{n}\left( \left[
\nabla_{\tau_{0,\sigma}}(R_{0})
(\tau_{0,\beta})\right]^{\wedge}\right) .
\end{equation}
For this, we  define functions $r_{\beta\mu}$ on $M$ by
$R_{0}(\tau_{0,\beta}) = r_{\beta \mu} \tau_{0,\mu}$. With this
notation, relation (\ref{required}) is proved as follows:
\begin{align*}
&\hat{\nabla}_{\tau_{0,\alpha}}(\hat{R}_{0})(T^{n}(\tau_{0,\beta}))=
\hat{\nabla}_{\tau_{0,\alpha}}\left( T^{n} \left[ R_{0}(\tau_{0,\beta})
\right]^{\wedge}\right) =\tau_{0,\alpha} \left( \widehat{r_{\beta \mu}}\right)
T^{n}(\tau_{0,\mu})\\
&= M_{\alpha}^{\sigma}\left[ \frac{\partial r_{\beta \mu}}{\partial
t_{0}^{\sigma}}\right]^{\wedge}T^{n}(\tau_{0,\mu}) = M^{\sigma}_{\alpha} T^{n}\left(\left[
\nabla_{\tau_{0,\sigma}}(R_{0})(\tau_{0,\beta})\right]^{\wedge}\right).
\end{align*}

Combining the above relation with (\ref{extindere-cr-1})  we obtain:
\begin{align*}
&\hat{\nabla}_{\tau_{0,\alpha}} (\hat{R}_{0}) \left(
T^{n}(\tau_{0,\beta})\right) + \hat{C}_{\tau_{0,\alpha}}\left(
T^{n} (\tau_{0,\beta})\right)\\
&  = M^{\sigma}_{\alpha} T^{n}\left(\left[
\nabla_{\tau_{0,\alpha}}(R_{0}) (\tau_{0,\beta}) +
C_{\tau_{0,\sigma}}(\tau_{0,\beta})\right]^{\wedge}\right) \\
&= M_{\alpha}^{\sigma} T^{n} \left(
[{C}_{\tau_{0,\sigma}},R_{\infty}](\tau_{0,\beta})^{\wedge}\right)
= [ \hat{C}_{\tau_{0,\alpha}}, \hat{R}_{\infty}]\left( T^{n}(\tau_{0,\beta})\right)
\end{align*}
where we used
$$
\nabla (R_{0}) + C = [C, R_{\infty}],
$$
because $(T^{1,0}M, \nabla , \eta , C, R_{0}, R_{\infty})$ is a
Saito bundle, and relation  (\ref{extindere-cr}).  We obtained (\ref{need-saito}), as required.

\end{proof}

In the followings we introduce the real structure $k$ into the
picture and we study the $tt^{*}$-geometry on the big phase space.

\subsection{Lifted harmonic Higgs bundles}\label{harmonic-sect}

\begin{thm}\label{thm-harmonic} Assume that $(T^{1,0}M, C, h)$ is a
harmonic Higgs bundle. Then
$(T^{1,0}M^{\infty},\hat{C}, \hat{h})$ is also a harmonic Higgs
bundle.
\end{thm}

The proof follows by combining the  Lemmas \ref{preliminar-1} and
\ref{preliminar-2} (see below). These lemmas hold for $h$ an
arbitrary pseudo-Hermitian metric on $M$ (not necessarily related
to $\eta$ by means of $k$), without any compatibility conditions
between $h$  and the Frobenius structure on $M$.

\begin{lem}\label{preliminar-1}
i) The following relation holds:
\begin{equation}\label{first-tt-1}
\left(\partial^{\hat{D}}\hat{C}\right)_{\tau_{0,\alpha},
\tau_{0,\beta}} \left(T^{n}(\tau_{0,\gamma}) \right) =
M^{\nu}_{\alpha} M^{\sigma}_{\beta} T^{n}\left( \left[
(\partial^{D}C)_{\tau_{0,\nu},\tau_{0,\sigma}}(\tau_{0,\gamma
})\right]^{\wedge}\right) ,
\end{equation}
for any $1\leq \alpha ,\beta ,\gamma\leq N$ and $n\geq 0.$ If $m$
or $p$ is bigger than zero, then
\begin{equation}\label{first-tt-2}
\left(\partial^{\hat{D}}\hat{C}\right)_{T^{m}(\tau_{0,\alpha}),
T^{p}(\tau_{0,\beta})} \left( T^{n}(\tau_{0,\gamma}) \right)=0.
\end{equation}
ii) In particular, if the first $tt^{*}$-equation
$\partial^{D}C=0$ for $(T^{1,0}M, C, h)$ holds,
then the first $tt^{*}$-equation $\partial^{\hat{D}}\hat{C}=0$ for
$(T^{1,0}M^{\infty},\hat{C}, \hat{h})$ holds as well.
\end{lem}

\begin{proof}
From the definition of $\hat{C}$ and
\begin{equation}
[T^{m} (\tau_{0,\alpha}), T^{p}(\tau_{0,\beta})] = 0,\quad\forall m,p\geq 1,
\end{equation}
it is immediately clear that (\ref{first-tt-2})
is true when both $m, p\geq 1$.
Also, note that $\hat{D}_{\mathcal W}=\hat{\nabla}_{\mathcal W}$, for any
vector field $\mathcal W$ in the image of $T$. Using this remark,
together with the definition of $\hat{C}$ again, we obtain:
$$
(\partial^{\hat{D}}\hat{C})_{\tau_{0,\alpha}, T^{p}(\tau_{0,\beta})}
\left( T^{n}(\tau_{0,\gamma})\right)
= (d^{\hat{\nabla}}\hat{C})_{\tau_{0,\alpha}, T^{p}(\tau_{0,\beta})}
\left( T^{n}(\tau_{0,\gamma})\right) ,\quad\forall p\geq 1,
$$
which is zero from Lemma \ref{lem-1}.
It remains to prove (\ref{first-tt-1}).
For this, we compute $\hat{D}_{\tau_{0,\alpha}}\left(
\hat{C}_{\tau_{0,\beta}}\right) \left( T^{n}(\tau_{0,\gamma})\right)$
as follows:
\begin{align*}
&\hat{D}_{\tau_{0,\alpha}}\left( \hat{C}_{\tau_{0,\beta}}\right)
\left( T^{n}(\tau_{0,\gamma})\right)=
\hat{D}_{\tau_{0,\alpha}}\left(
\hat{C}_{\tau_{0,\beta}}T^{n}(\tau_{0,\gamma})\right)
-\hat{C}_{\tau_{0,\beta}} \hat{D}_{\tau_{0,\alpha}} \left(
T^{n}(\tau_{0,\gamma})\right)\\
&= \hat{D}_{\tau_{0,\alpha}}\left( M^{\sigma}_{\beta} T^{n} \left[
C_{\tau_{0,\sigma}}(\tau_{0,\gamma})\right]^{\wedge}\right)  -
M^{\sigma}_{\alpha} \hat{C}_{\tau_{0,\beta}} \left( T^{n}\left[
D_{\tau_{0,\sigma}}(\tau_{0,\gamma})\right]^{\wedge}\right)\\
&= \eta^{\sigma \nu} <<
\tau_{0,1}\tau_{0,\nu}\tau_{0,\beta}\tau_{0,\alpha} >>_{{}_0} T^{n}
\left( \left[
C_{\tau_{0,\sigma}}(\tau_{0,\gamma})\right]^{\wedge}\right) \\
& + M_{\beta}^{\sigma} \hat{D}_{\tau_{0,\alpha}} \left( T^{n}
\left[ C_{\tau_{0,\sigma}}(\tau_{0,\gamma}) \right]^{\wedge}
\right)  - M_{\alpha}^{\sigma} \hat{C}_{\tau_{0,\beta}}\left(
T^{n} \left[ D_{\tau_{0,\sigma}}(\tau_{0,\gamma}) \right]^{\wedge}
\right) ,
\end{align*}
where in the third line we used
$$
\tau_{0,\alpha} (M_{\beta}^{\sigma}) =\eta^{\sigma \nu} <<
\tau_{0,1}\tau_{0,\nu}\tau_{0,\beta}\tau_{0,\alpha} >>_{{}_0}.
$$
To simplify notation, we define
\begin{align*}
E_{1}(\alpha , \beta ,\gamma ,n)&:=\eta^{\sigma \nu} <<
\tau_{0,1}\tau_{0,\nu}\tau_{0,\beta}\tau_{0,\alpha} >>_{{}_0} T^{n}
\left( \left[
C_{\tau_{0,\sigma}}(\tau_{0,\gamma})\right]^{\wedge}\right)\\
E_{2}(\alpha , \beta ,\gamma ,n) &:= M_{\beta}^{\sigma}
\hat{D}_{\tau_{0,\alpha}}\left( T^{n} \left[
C_{\tau_{0,\sigma}}(\tau_{0,\gamma}) \right]^{\wedge}\right) =
M^{\sigma}_{\beta}\hat{D}_{\tau_{0,\alpha}}\left(
\widehat{c_{\sigma\gamma}^{\mu}}T^{n}(\tau_{0,\mu}) \right)\\
E_{3}(\alpha , \beta ,\gamma ,n) &:= M_{\alpha}^{\sigma}
\hat{C}_{\tau_{0,\beta}}\left( T^{n} \left[
D_{\tau_{0,\sigma}}(\tau_{0,\gamma}) \right]^{\wedge} \right) ,
\end{align*}
where $c_{\sigma\gamma}^{\mu}$ are the structure constants of the
Frobenius multiplication on $M$, already defined in (\ref{constanta}). With
these
$$
\hat{D}_{\tau_{0,\alpha}}\left( \hat{C}_{\tau_{0,\beta}}\right)
\left( T^{n}(\tau_{0,\gamma})\right)= (E_{1}+ E_{2} -
E_{3})(\alpha , \beta , \gamma , n).
$$
Since
$$
(\partial^{\hat{D}}\hat{C})_{\tau_{0,\alpha}, \tau_{0,\beta}}
\left( T^{n}(\tau_{0,\gamma})\right)  =
\hat{D}_{\tau_{0,\alpha}}\left(\hat{C}_{\tau_{0,\beta}}\right)
(T^{n}(\tau_{0,\gamma}))
-\hat{D}_{\tau_{0,\beta}}\left(\hat{C}_{\tau_{0,\alpha}}\right)
(T^{n}(\tau_{0,\gamma})) ,
$$
we need to compute the skew part (in $\alpha$ and $\beta$) of
$E_{1}$, $E_{2}$ and $E_{3}.$ It is clear that $E_{1}$ is
symmetric in $\alpha$ and $\beta$. Also,
\begin{align*}
E_{2}(\alpha ,\beta ,\gamma ,n)&= M_{\beta}^{\sigma} \left(
\tau_{0,\alpha}\left(\widehat{c_{\sigma\gamma}^{\mu}}\right)
T^{n}(\tau_{0,\mu}) + \widehat{c_{\sigma\gamma}^{\mu}}
M_{\alpha}^{\nu}
T^{n}\left[ D_{\tau_{0,\nu}}(\tau_{0,\mu})\right]^{\wedge} \right)\\
& = M_{\beta}^{\sigma} \left( \left[ \frac{\partial
c_{\sigma\gamma}^{\mu}}{\partial t^{\nu}_{0}}\right]^{\wedge}
M_{\alpha}^{\nu} T^{n} (\tau_{0,\mu}) +
\widehat{c_{\sigma\gamma}^{\mu}} M_{\alpha}^{\nu} T^{n}\left[
D_{\tau_{0,\nu}}(\tau_{0,\mu}) \right]^{\wedge}\right)
\end{align*}
and a straightforward computation shows that
\begin{align*}
&E_{2}(\alpha ,\beta ,\gamma ,n) - E_{2}(\beta,\alpha ,\gamma ,n)=
M^{\sigma}_{\beta} M^{\nu}_{\alpha} \left[ \frac{\partial
c_{\sigma\gamma}^{\mu}}{\partial t_{0}^{\nu}} - \frac{\partial
c_{\nu\gamma}^{\mu}}{\partial t_{0}^{\sigma}}\right]^{\wedge}
T^{n}(\tau_{0,\mu})\\
& + M_{\beta}^{\sigma} M_{\alpha}^{\nu} T^{n}\left( \left[
c_{\sigma\gamma}^{\mu} D_{\tau_{0,\nu}}(\tau_{0,\mu}) -
c_{\nu\gamma}^{\mu} D_{\tau_{0,\sigma}} (\tau_{0,\mu})
\right]^{\wedge} \right) .
\end{align*}
Since $\frac{\partial
c_{\sigma\gamma}^{\mu}}{\partial t^{\nu}_{0}}$ is symmetric in
$\sigma$ and $\nu$, the first term in the right hand side of the above
relation vanishes  and we obtain:
\begin{equation}
E_{2}(\alpha ,\beta ,\gamma ,n) - E_{2}(\beta,\alpha ,\gamma ,n)=
M_{\beta}^{\sigma} M_{\alpha}^{\nu}\left[
c_{\sigma\gamma}^{\mu}f_{\nu \mu}^{\delta} -
c_{\nu\gamma}^{\mu} f_{\sigma \mu}^{\delta}
\right]^{\wedge}
T^{n}(\tau_{0,\delta}),
\end{equation}
where the functions $f_{\mu\gamma}^{\delta}$ are defined by
$D_{\tau_{0,\mu}}(\tau_{0,\gamma }) = f_{\mu\gamma}^{\delta}\tau_{0,\delta}$. A
similar computation shows that
\begin{equation}
E_{3}(\alpha ,\beta ,\gamma ,n) - E_{3}(\beta,\alpha ,\gamma ,n)=
M^{\nu}_{\alpha} M^{\sigma}_{\beta} \left[
c_{\sigma \mu}^{\delta} f^{\mu}_{\nu\gamma} -
c_{\nu \mu}^{\delta}f_{\sigma\gamma}^{\mu}
\right]^{\wedge} T^{n}(\tau_{0,\delta}) .
\end{equation}
Hence we obtain
\begin{align*}
& (\partial^{\hat{D}}\hat{C})_{\tau_{0,\alpha}, \tau_{0,\beta}}
\left( T^{n}(\tau_{0,\gamma})\right)
= (E_{2}- E_{3})(\alpha , \beta , \gamma, n)
- (E_{2}- E_{3})(\beta , \alpha , \gamma, n)\\
& = M^{\nu}_{\alpha} M^{\sigma}_{\beta} \left[
c_{\sigma\gamma}^{\mu}f_{\nu \mu}^{\delta} -
c_{\nu\gamma}^{\mu} f_{\sigma \mu}^{\delta} -
c_{\sigma \mu}^{\delta} f_{\nu\gamma}^{\mu} +
c_{\nu \mu}^{\delta}f_{\sigma\gamma}^{\mu}
\right]^{\wedge} T^{n}(\tau_{0,\delta}).
\end{align*}
On the other hand, one may check that
$$
(\partial^{D} C)_{\tau_{0,\nu},\tau_{0,\sigma}} (\tau_{0,\gamma })
=\left(
c_{\sigma\gamma}^{\mu}f_{\nu \mu}^{\delta} -
c_{\nu\gamma}^{\mu} f_{\sigma \mu}^{\delta} -
c_{\sigma \mu}^{\delta} f_{\nu\gamma}^{\mu} +
c_{\nu \mu}^{\delta}f_{\sigma\gamma}^{\mu}
\right) \tau_{0,\delta} .
$$
Combining the above two relations, we obtain (\ref{first-tt-1}),
as required. Our claim follows.
\end{proof}

\begin{lem}\label{preliminar-2}
i) The following relation holds: for any $1\leq \alpha , \beta ,
\gamma\leq N$ and $m\geq 0$,
\begin{align}
\nonumber&{}^{\hat{D}} \!
R_{\tau_{0,\alpha},\overline{\tau_{0,\beta}}}
(T^{m}(\tau_{0,\gamma} )) + [\hat{C}_{\tau_{0,\alpha}},
\hat{C}^{\dagger}_{\overline{\tau_{0,\beta}}}](T^{m}(\tau_{0,\gamma}))\\
\label{sec-c}&=M^{\sigma}_{\alpha}\overline{ M^{\nu}_{\beta}}
T^{m}\left(\left[ {}^{D} \!
R_{\tau_{0,\sigma},\overline{\tau_{0,\nu}}}(\tau_{0,\gamma}) +
[C_{\tau_{0,\sigma}},
C^{\dagger}_{\overline{\tau_{0,\nu}}}](\tau_{0,\gamma})
\right]^{\wedge}\right)\,,
\end{align}
where $C^{\dagger}\in \Omega^{0,1}(M, \mathrm{End}T^{1,0}M)$ and
$\hat{C}^{\dagger}\in \Omega^{0,1}(M^{\infty},
\mathrm{End}T^{1,0}M^{\infty})$ are, respectively, the $h$-adjoint of $C$ and
the $\hat{h}$-adjoint of $\hat{C}.$

ii) In particular, if the second $tt^{*}$-equation
$$
{}^{D}\! R + [C, C^{\dagger}]=0
$$
for $(T^{1,0}M, C, h)$ holds, then
the second $tt^{*}$-equation
$$
{}^{\hat{D}} \! R + [\hat{C}, \hat{C}^{\dagger}]=0
$$
for $(T^{1,0}M^{\infty},
\hat{C}, \hat{h})$ holds as well.
\end{lem}

\begin{proof}
We first compute the $\hat{h}$-adjoint
$\hat{C}^{\dagger}_{\overline{{\tau}_{0,\beta}}}$ of
$\hat{C}_{\tau_{0,\beta}}$, as follows:
\begin{align*}
&\hat{h}\left(\hat{C}_{{\tau}_{0,\beta}}\left(
T^{n}(\tau_{0,\alpha})\right), T^{m} (\tau_{0,\gamma})\right)=
M_{\beta}^{\sigma} \hat{h} \left( T^{n}\left(
 \left[ C_{\tau_{0,\sigma}}(\tau_{0,\alpha})
\right]^{\wedge}\right) , T^{m}(\tau_{0,\gamma})\right)\\
&= M_{\beta}^{\sigma}\delta_{nm} \left[ h\left(
C_{\tau_{0,\sigma}}(\tau_{0,\alpha}),
\tau_{0,\gamma}\right)\right]^{\wedge}
= M_{\beta}^{\sigma}\delta_{nm}\left[ h\left( \tau_{0,\alpha}, C^{\dagger}_{\overline{{\tau}_{0,\sigma}}}(\tau_{0,\gamma})\right)\right]^{\wedge}\\
&= M_{\beta}^{\sigma}\hat{h}\left( T^{n}(\tau_{0,\alpha}),T^{m}
\left( \left[ C^{\dagger}_{\overline{\tau_{0,\sigma}}}(\tau_{0,\gamma})
\right]^{\wedge}\right) \right)\\
&=
\hat{h} \left(
T^{n}(\tau_{0,\alpha}),\overline{M_{\beta}^{\sigma}} T^{m}\left(\left[
C^{\dagger}_{\overline{{\tau}_{0,\sigma}}}(\tau_{0,\gamma})
\right]^{\wedge}\right)\right).
\end{align*}
We obtain:
$$
\hat{C}^{\dagger}_{\overline{{\tau}_{0,\beta}}} \left(
T^{m}(\tau_{0,\gamma}) \right) = \overline{M_{\beta}^{\sigma}}
T^{m}\left(\left[
C^{\dagger}_{\overline{{\tau}_{0,\sigma}}}(\tau_
{0,\gamma})\right]^{\wedge}\right) ,
$$
which, combined with
$$
\hat{C}_{\tau_{0,\alpha}} \left( T^{m} (\tau_{0,\gamma })\right) =
M^{\sigma}_{\alpha} T^{m}\left( [ C_{\tau_{0,\sigma}} (\tau_{0,\gamma
})]^{\wedge}\right)
$$
gives
$$
[ \hat{C}_{{\tau}_{0,\alpha}}
,\hat{C}^{\dagger}_{\overline{{\tau}_{0,\beta}}}] \left(
T^{m}(\tau_{0,\gamma})\right)  = M_{\alpha}^{\sigma}
\overline{M_{\beta}^{\nu}} T^{m}\left( [C_{{\tau}_{0,\sigma}},
C^{\dagger}_{\overline{{\tau}_{0,\nu}}}]
(\tau_{0,\gamma})^{\wedge}\right) .
$$
The above relation, together with
$$
{}^{\hat{D}}\! R_{\tau_{0,\alpha}, \overline{\tau_{0,\beta}}}
\left( T^{m}(\tau_{0,\gamma })\right) = M_{\alpha}^{\sigma}
\overline{M_{\beta}^{\nu}} T^{m}\left( \left[ {}^{{D}}\!
R_{\tau_{0,\sigma}, \overline{\tau_{0,\nu}} }(\tau_{0,\gamma
})\right]^{\wedge}\right)
$$
(see Lemma \ref{curvature}), implies (\ref{sec-c}). This proves
{\it i)}. Claim {\it ii)} follows  from {\it i)} and
$$
{}^{\hat{D}}\! R_{\mathcal W_{1}, \mathcal W_{2}} =0,\quad [
\hat{C}_{\mathcal W_{1}}, \hat{C}^{\dagger}_{\bar{\mathcal
W}_{2}}]=0
$$
which hold when $\mathcal W_{1}$ or $\mathcal W_{2}$ belongs to
the image of $T$.

\end{proof}

\subsection{Lifted harmonic real Saito bundles}\label{extended-sect}

\begin{thm}\label{saito-thm}
Assume that $(T^{1,0}M, \nabla , \eta , R_{0}, R_{\infty}, k)$ is
a real Saito bundle (respectively, a harmonic real Saito bundle).
Then $(T^{1,0}M^{\infty},\hat{\nabla}, \hat{\eta}, \hat{C},
\hat{R}_{0}, \hat{R}_{\infty},\hat{k})$ is also a real Saito
bundle  (respectively, a harmonic real Saito bundle).

\end{thm}

The above theorem is a consequence of Theorems \ref{thm-saito} and
\ref{thm-harmonic} and the following lemma.

\begin{lem}\label{compat-lem} If $h$ is compatible with $\eta$, then $\hat{h}$ is compatible
with $\hat{\eta}.$\end{lem}

\begin{proof}
Recall that $h$ is compatible with $\eta$ if the Chern connection
$D$ of $h$ preserves $\eta $, i.e. $D(\eta )=0$ and the
compatibility of $\hat{h}$ and $\hat{\eta}$ is defined in a
similar way. In order to prove our claim, we will show that

\begin{equation}\label{impreuna-1}
\hat{D}_{\mathcal W} (\hat{\eta}) = 0, \quad\forall {\mathcal
W}\in \mathrm{Im}(T)
\end{equation}
and
\begin{equation}\label{impreuna-2}
\hat{D}_{\tau_{0,\gamma}} (\hat{\eta}) \left( T^{n}
(\tau_{0,\alpha}), T^{m}(\tau_{0,\beta})\right)=
M^{\sigma}_{\gamma}\delta_{nm} \left[ D_{\tau_{0,\sigma}}(\eta )
(\tau_{0,\alpha}, \tau_{0,\beta})\right]^{\wedge}.
\end{equation}
To prove these two relations, we notice that for any
vector field ${\mathcal V}\in {\mathcal T}^{1,0}_{M^{\infty}}$ and
$n,m\geq 0$, $1\leq \alpha ,\beta \leq N$,
\begin{align}
\nonumber\hat{D}_{\mathcal V}(\hat{\eta})\left( T^{n}
(\tau_{0,\alpha}), T^{m}(\tau_{0,\beta})\right)& =
-\hat{\eta}\left( \hat{D}_{\mathcal V} \left( T^{n}(\tau_{0,\alpha})\right), T^{m}(\tau_{0,\beta})\right)\\
\label{h-d}& - \hat{\eta} \left( T^{n} ({\tau_{0,\alpha}}),
\hat{D}_{\mathcal V} \left( T^{m}(\tau_{0,\beta})\right)\right) ,
\end{align}
because
$\hat{\eta}\left( T^{n} (\tau_{0,\alpha}),
T^{m}(\tau_{0,\beta})\right)$ is constant.
This relation, together with the expression of $\hat{D}$
from Lemma \ref{chern}, implies
(\ref{impreuna-1}).  Letting ${\mathcal V}:=
\tau_{0,\gamma}$ in (\ref{h-d}) we obtain:
\begin{align*}
&\hat{D}_{\tau_{0,\gamma}} (\hat{\eta})
\left( T^{n} (\tau_{0,\alpha}), T^{m}(\tau_{0,\beta})\right) =\\
&= -M^{\sigma}_{\gamma}\hat{\eta}\left( T^{n}\left( \left[
{D}_{\tau_{0,\sigma}}(\tau_{0,\alpha})\right]^{\wedge}\right) ,
T^{m}(\tau_{0,\beta})\right)
- M^{\sigma}_{\gamma}\hat{\eta} \left( T^{n} (\tau_{0,\alpha}),
T^{m}\left( \left[ {D}_{\tau_{0,\sigma}}(\tau_{0,\beta})\right]^{\wedge}\right)
 \right)\\
&= -M^{\sigma}_{\gamma} \delta_{mn}\left[  \eta \left(
D_{\tau_{0,\sigma}}(\tau_{0,\alpha}),\tau_{0,\beta}\right)\right]^{\wedge}
- M^{\sigma}_{\gamma}\delta_{mn} \left[ \eta\left(
\tau_{0,\alpha}, D_{\tau_{0,\sigma}}
(\tau_{0,\beta})\right)\right]^{\wedge}\\
&= M^{\sigma}_{\gamma}\delta_{nm} \left[ D_{\tau_{0,\sigma}}(\eta )
(\tau_{0,\alpha}, \tau_{0,\beta})\right]^{\wedge},\\
\end{align*}
i.e. relation (\ref{impreuna-2}) holds as well.
\end{proof}

\begin{rem}\label{dchk}{\rm The above lemma, combined with
Theorem \ref{thm-harmonic}, imply that if $(T^{1,0}M,D,C,h,k)$ is a $DChk$-bundle,
then $(T^{1,0}M^{\infty},\hat{D},\hat{C},\hat{h},\hat{k})$ is a $\hat{D}\hat{C}\hat{h}\hat{k}$-bundle
 (the definition of $DChk$-bundles was recalled in Remark \ref{CV}).}
\end{rem}

\subsection{Lifted harmonic potential real Saito bundles}\label{extended-real-sect}

\begin{thm}\label{potential-thm}
Suppose that $(M, \bullet , e, \eta , E)$ is a harmonic Frobenius
manifold, with real structure $k$ and potential $A$. Then
$T^{1,0}M^{\infty}$, with the data $(\hat{\nabla}, \hat{\eta},
\hat{C}, \hat{R}_{0}, \hat{R}_{\infty}, \hat{k}, \hat{A})$, is  a
harmonic potential real Saito bundle.
\end{thm}

\begin{proof}
It is easy to check that taking Hermitian adjoints of endomorphisms, with respect to $h$ and $\hat{h}$, commutes
with taking natural lifts, so that $(\hat{A})^{\dagger} = \widehat{A^{\dagger}}$
(we use the same symbol ${}^\dagger$ to denote $h$ and $\hat{h}$-adjoints). There is a similar commutativity property
when the metrics $\eta$ and $\hat{\eta}$ are used instead of $h$ and $\hat{h}.$
Thus, since $A$ is $\eta$-self adjoint, $\hat{A}$ is
$\hat{\eta}$-self adjoint. Similarly,
$$
\hat{R}_{\infty} + [(\hat{A})^{\dagger}, \hat{R}_{0}]= \left(
{R}_{\infty} + [A^{\dagger}, R_{0}]\right)^{\wedge}
$$
is $\hat{h}$-self adjoint, because  $R_{\infty} +[A^{\dagger},
R_{0}]$ is $h$-self adjoint.  From Theorem \ref{saito-thm}, we
only have to check that
\begin{equation}\label{doua}
\hat{D}^{(1,0)}= \hat{\nabla}
- [\widehat{{A}^{\dagger}},\hat{C}],\quad \hat{D}^{(1,0)}\hat{A}=\hat{C}.
\end{equation}
Using that $T^{m}(\tau_{0,\beta})$
are $\hat{\nabla}$-parallel, the first equality (\ref{doua}) is
equivalent to
\begin{equation}\label{t-n-sec}
\hat{D}_{T^{n}(\tau_{0,\alpha} )}
\left(T^{m}(\tau_{0,\beta})\right) +
[\widehat{A^{\dagger}},\hat{C}_{T^{n}(\tau_{0,\alpha} )} ]\left(
T^{m}(\tau_{0,\beta})\right) =0.
\end{equation}
When $n>0$ both terms of (\ref{t-n-sec})  are zero. We now show
that (\ref{t-n-sec}) holds also for $n=0.$ For this, we use
\begin{equation}\label{pot1}
\hat{D}_{\tau_{0,\alpha}}
\left(T^{m}(\tau_{0,\beta})\right) = M_{\alpha}^{\sigma}
T^{m}\left(\left[ D_{\tau_{0,\sigma}}(\tau_{0,\beta})\right]^{\wedge}\right)
\end{equation}
and
\begin{equation}\label{fin-art}
[\widehat{A^{\dagger}},\hat{C}_{\tau_{0,\alpha}}]\left(
T^{m}(\tau_{0,\beta})\right)  =
M_{\alpha}^{\sigma}T^{m} \left( [A^{\dagger},
C_{\tau_{0,\sigma}}](\tau_{0,\beta})^{\wedge}\right)
\end{equation}
(easy check). On the other hand, $\nabla = D + [A^{\dagger}, C]$
($M$ with the given data is a harmonic Frobenius manifold, with
potential $A$) and, since $\nabla (\tau_{0,\beta})=0$,
\begin{equation}\label{pot2}
D_{\tau_{0,\sigma}}(\tau_{0,\beta}) + [A^{\dagger}, C_{\tau_{0,\sigma}}]
(\tau_{0,\beta})=0.
\end{equation}
Combining (\ref{pot1}), (\ref{fin-art}) and (\ref{pot2}) we
obtain
$$
\hat{D}_{\tau_{0,\alpha}}\left( T^{m}(\tau_{0,\beta})\right)
+ [\widehat{A^{\dagger}},\hat{C}_{\tau_{0,\alpha}}]
\left( T^{m}(\tau_{0,\beta})\right)  =0.
$$
The first equality (\ref{doua}) is proved. The second equality
(\ref{doua}) can be proved equally easy, by using
\begin{align}
\nonumber&\hat{D}_{\tau_{0,\alpha}}(\hat{A})\left( T^{m}(\tau_{0,\beta})\right)
= M^{\sigma}_{\alpha} T^{m}\left( \left[ D_{\tau_{0,\sigma}}(A)
(\tau_{0,\beta})\right]^{\wedge}\right)\\
\label{fin-art-1} &\hat{D}_{T^{n}(\tau_{0,\alpha})}(\hat{A})
\left( T^{m}(\tau_{0,\beta})\right)
= 0,
\end{align}
the definition of $\hat{C}$ and the condition $D^{1,0}A= C.$

\end{proof}

\begin{rem} {\rm Relations (\ref{fin-art}) and (\ref{fin-art-1})
remain true when ${A}^{\dagger}$, respectively ${A}$, are
replaced by  any endomorphism of $T^{1,0}M.$
Combining these facts with Remark \ref{dchk},
it is easy to see that if $(T^{1,0}M,{C}, {h}, k,
{\mathcal U},{\mathcal Q})$ is a CV-bundle (see Remark \ref{CV}),
then
$ (T^{1,0}M^{\infty},\hat{C}, \hat{h}, \hat{k}, \hat{\mathcal U},\hat{\mathcal Q})$
is also a CV-bundle.}
\end{rem}

\section{Discussion}\label{discussion-str}

The construction of these Hermitian structures on the big phase space rests on the pre-existence of
two different structures:
\begin{itemize}
\item[$\bullet$] the lifting map $u^\alpha = \eta^{\alpha\beta} << \tau_{0,1} \tau_{0,\beta} >>_{{}_0}\,;$
\item[$\bullet$] the Hermitian structures on the small phase space,
\end{itemize}
and one should comment separably on the tractability of each of these two points.

The first point rests on the work of Dijkgraaf and Witten \cite{DijkgraafWitten}. We repeat verbatim their construction, using their
normalization of the topological recursion relation which differs by a factor to the one used in \cite{liu1,liu2} and in the above. The
string equation gives (where $u_{\alpha}=\eta_{\alpha\beta} u^\beta$ and $t_{i,\alpha}=\eta_{\alpha\beta}t^\beta_i$)
\[
u_\alpha = t_{0,\alpha} + \sum_{i=0}^\infty (i+1) t_{i+1,\beta} << \tau_i(\gamma_\beta) \gamma_\alpha >>_{{}_0}
\]
and small phase space calculations give the constitutive relations
\[
<< \tau_i(\gamma_\beta) \gamma_\alpha >>_{{}_0} = R_{\alpha,\beta,i}(u_1\,,\ldots\,,u_N)\,.
\]
Combining these gives
\[
u_\alpha = t_{0,\alpha} + \sum_{i=0}^\infty (i+1) t_{i+1,\beta}R_{\alpha,\beta,i}(u_1\,,\ldots\,,u_N)\,.
\]
Inverting these equations gives
$u_\alpha=u_\alpha( t^n_\beta )\,,$ and hence the two-point
correlation functions $<< \tau_{0,1} \tau_{0,\beta} >>_{{}_0}$ as
functions of the big phase space variables, as required in order
to perform the canonical lift described in Section
\ref{higgs-str}\,. If $\mathrm{dim}(M)=1$ this reduces to the
single equation
\begin{equation}
u=t_0 + \sum_{i=1}^\infty t_i u^i\,.
\label{2Dgravity}
\end{equation}
Thus the extension to the big phase space is entirely tractable, up to the inversion of these equations.

Hermitian structures in one dimension are trivial. One just has
$h(\partial_0,\partial_0)=|a(t)|$ for some non-vanishing function
$a(t)\,,$ and the corresponding anti-holomorphic involution is
then given by $k(\partial_0) = a^{-1}|a|\,\partial_0\,.$ In fact
this gives a positive CDV-structure \cite{takahashi}\,. Thus
combining this with the solution to (\ref{2Dgravity}) gives
Hermitian structures on the big phase space to 2 dimensional
gravity \cite{witten1}\,.

The existence of $tt^{*}$-geometries in higher dimensions is harder. The Hermitian structures involve the construction of solutions to Toda and harmonic map type equations \cite{cecotti,guest},
often with very specific boundary conditions. Even in dimension two this involves specific solutions to the Painlev\'e III equation. While these equations are
integrable, one is rapidly drawn into very sophisticated isomonodromy problems. For details of these constructions see, for examples originating in quantum cohomology, \cite{iritani}, and
for examples originating in singularity theory, \cite{hert1,sabbah-art}\,.

Ultimately these constructions rely on the integrability properties of the $tt^*$-equations \cite{cecotti,dubrovin3}. The $tt^*$-equations are the compatibility conditions, or zero-curvature
equations, for the deformed connections
\begin{eqnarray*}
{}^{(\lambda)}\!D_X\,Y & = & (D_X- \lambda C_X)Y\,,\\
{}^{(\lambda)}\!D_{\bar X}\,Y & = & (D_{\bar{X}} - \lambda^{-1} C_X^\dagger)Y\,.
\end{eqnarray*}
It is easy to show that the $\lambda^{\pm 2}$-terms in the curvature of
this connection vanish from the properties of the Higgs field,
the $\lambda^{\pm 1}$-terms vanish
from the first $tt^*$-equation $\partial^DC=0\,,$ and the
$\lambda^{0}$-term vanishes from the second $tt^*$-equation ${}^D\!R+[C,C^\dagger]=0$. They thus
provide a Lax pair for, and hence the integrability of, the $tt^*$-equations.

The results in this paper show that there exists a solution of the analogous Lax-pair on the big phase space $M^\infty.$ However, while one
may speculate that they define an integrable system on $M^\infty$ all that has been shown is that any solution on $M$ may be naturally lifted to a solution
on $M^\infty\,,$ not that all solutions arise in this way. The integrability aspects of the $tt^*$-equations on the big phase space deserves a separate study.

One geometric structure that plays a prominent role in quantum cohomology is the Euler field, which on the big phase space takes the form
\[
\chi=-\sum_{m,\alpha} (m+b_\alpha-b_1-1) {\tilde t}_m^\alpha
\tau_m(\gamma_\alpha) - \sum_{m,\alpha,\beta}
\mathcal{C}^\beta_\alpha {\tilde t}^\alpha_m
\tau_{m-1}(\gamma_\beta)\,,
\]
with the associated quasihomogeneity equation
\[
<<\chi>>_g=2(b_1+1)(1-g) {\mathcal{F}_g} + \frac{1}{2} \delta_{g,0} \sum_{\alpha,\beta} \mathcal{C}_{\alpha\beta} t^\alpha_0 t^\beta_0 - \frac{1}{24} \int_V c_1(V) \cup c_{d-1}(V)\,
\]
(for a precise definition of the various constants, see
\cite{liu1,liu2}). To develop further these ideas one should study
the homogeneity properties of the lifted objects on the big phase
space, starting with their homogeneity properties on the small
phase space. Another direction of research would be to reformulate
these constructions in the semi-simple case in terms of
idempotents and canonical coordinates, following \cite{liu4}.
However, a more geometric problem is to seek a description of these
Hermitian structures in terms of Givental's Lagrangian cones
\cite{givental}. Recall that Givental showed that the function
${\mathcal{F}}_0$ satisfies the Topological Recursion Relations, the String
Equation and the Dilaton Equation if and only if a corresponding
Lagrangian submanifold has certain natural geometric properties.
This gives a beautiful interpretation of quantum cohomology and
leads naturally to the quantization of these objects.
Understanding how the (compatible) Hermitian structures introduced
in this paper can be interpreted within this framework would be of
great interest. In a sense the results of this paper could be seen as a \lq pre-Givental\rq~approach. Understanding the symmetries and quantization of
the $tt^*$-equations, following Givental, would provide an elegant
solution to the problems addressed in this paper. We hope to address such problems in subsequent
work.

\end{document}